\documentclass{article}

\usepackage{amsmath}
\usepackage{amsthm}
\usepackage{amsfonts}
\usepackage{amsbsy}
\usepackage{amssymb}
\usepackage[initials]{amsrefs}
\newtheorem{theorem}{Theorem}

\newtheorem{lemma}{Lemma}

\newcommand{\R}{{\mathbb R}}
\newcommand{\Z}{{\mathbb Z}}

\newcommand{\C}{{\mathbb C}}

\newcommand{\set}[2]{ \left\{ #1 \ \left| \ #2 \right. \right\} }

\newcommand{\ang}[1]{\left< #1 \right>}
\newcommand{\diam}{{\mathrm{diam}~\!}}

\newcommand{\threesp}{{\mathcal M}_3}
\newcommand{\onesp}{{\mathcal M}}
\newcommand{\Lsp}{{\mathcal L}}
\newcommand{\Rsp}{{\mathcal R}}
\newcommand{\inc}{{\mathcal E}}
\newcommand{\dist}{{\mathrm{dist}}}

\title{Generalized curvature for certain Radon-like operators of intermediate dimension}
\author{Philip T. Gressman\footnote{Partially supported by NSF Grant DMS-1361697.}}

\begin{document}

\maketitle

\begin{abstract}
This paper establishes $L^p$-improving estimates for a variety of Radon-like transforms which integrate functions over submanifolds of intermediate dimension. In each case, the results rely on a unique notion of curvature which relates to, but is distinct from, Phong-Stein rotational curvature. The results obtained are sharp up to the loss of endpoints. The methods used are a new adaptation of the familiar method of inflation developed by Christ and others. Unlike most previous instances of this method, the present application does not require any particular linear algebraic relations to hold for the dimension and codimension.
\end{abstract}

\tableofcontents

\section{Introduction}
\subsection{Quadratic model cases}

The purpose of this paper is to develop broadly-applicable methods to understand the $L^p-L^q$ mapping properties of geometric averaging operators which are truly intermediate-dimensional, meaning that they are not well-understood by existing arguments applicable to either curves or hypersurfaces. The methods are based on an $L^p$-adapted version of the method of $TT^*T$, which is itself derived using ideas from Christ's method of inflation, first introduced in the study of the corkscrew curve \cite{christ1998}. The resulting arguments can be successfully applied in a wide variety of settings, in contrast to most earlier work in intermediate dimensions, which tends to be limited to analyses of isolated special cases. Here a wide variety of cases means, for example, that there are no combinatorial constraints on the dimensions and codimensions of the manifolds involved.  More precisely, the methods developed here were specifically intended to overcome the fairly common problem in inflation-type arguments in which certain integer quantities related to dimension need to have some precise factorization properties in order to proceed. Some important special cases whose analysis is unified by this new common framework include the following operators:
\begin{itemize}
\item Convolution with the maximal quadratic surface in $\R^5$: For every $t := (t_1,t_2) \in \R^2$, let $\gamma_M(t) := (t_1,t_2, t_1^2, 2t_1 t_2, t_2^2)$ and define
\begin{equation} R f(x) := \int_{[-1,1]^2} f( x + \gamma_M(t)) dt \label{modelmaxml} \end{equation}
for $x \in \R^5$ and $f$ any sufficiently regular function on $\R^5$. This operator belongs to an explicit family studied by Ricci \cite{ricci1997} using Fourier-analytic methods, in contrast to the geometric-combinatorial methods used here.
\item Convolution with the maximal quadratic surface in $\C^5$: For every $t := (t_1,t_2) \in \C^2$, let $\gamma_{M_\C}(t) := (t_1,t_2, t_1^2, 2t_1 t_2, t_2^2)$ and define
\begin{equation} R_\C f(x) := \int_{|t_1| + |t_2| \leq 1} f( x + \gamma_{M_C}(t)) |dt \wedge d \overline{t}| \label{modelmaxmlc} \end{equation}
for $x \in \C^5$ and $f$ any sufficiently regular function on $\C^5$.
\item Convolution with the harmonic quadratic $3$-surface in $\R^8$: For $t  \in \R^3$, let $\gamma_H(t) := (t_1,t_2,t_3, t_1^2 - t_2^2, t_2^2 - t_3^2 , 2 t_1 t_2, 2 t_2 t_3, 2 t_1 t_3)$. Define
\begin{equation} S f(x) := \int_{[-1,1]^3} f(x + \gamma_H(t)) dt \label{modelharmonic} \end{equation}
for $x \in \R^8$ and all sufficiently regular functions $f$ on $\R^8$.
\item Asymmetric averages over half-dimensional subspaces of $\R^{2n}$: For any fixed natural number $n \geq 1$, define
\begin{equation} T f (y,t) := \int_{[-1,1]^n} f(x, y + t x) dx \label{modelasym} \end{equation}
for $y \in \R^n$, $t \in \R$, and $f$ any sufficiently regular function on $\R^{n} \times \R^{n}$.
\end{itemize}
These four examples were explicitly chosen to demonstrate that, indeed, no specific affine linear relationship must hold among the dimensions of the input variables, output variables, and integration variables. In each case, the full range of $L^p-L^q$ estimates may be attained except for the extremal cases:
\begin{theorem}
Regarding the operators given by \eqref{modelmaxml}--\eqref{modelasym}:  \label{modelthm}
\begin{itemize}
\item The operator  \eqref{modelmaxml} is bounded from $L^p(\R^5)$ into $L^q(\R^5)$ provided that $(\frac{1}{p}, \frac{1}{q})$ is in the interior of the triangle with vertices $(0,0)$, $(1,1)$, and $( \frac{5}{8}, \frac{3}{8})$.
\item The operator  \eqref{modelmaxmlc} is bounded from $L^p(\C^5)$ into $L^q(\C^5)$ provided that $(\frac{1}{p}, \frac{1}{q})$ is in the interior of the triangle with vertices $(0,0)$, $(1,1)$, and $( \frac{5}{8}, \frac{3}{8})$.
\item The operator \eqref{modelharmonic} is bounded from $L^p(\R^8)$ into $L^q(\R^8)$ provided that $(\frac{1}{p}, \frac{1}{q})$ is in the interior of the triangle with vertices $(0,0)$, $(1,1)$, and $(\frac{8}{13}, \frac{5}{13})$.
\item For any compact set $\Omega \subset \R^{n+1}$, the operator \eqref{modelasym} maps $L^p(\R^{2n})$ into $L^q(\Omega)$ provided that $(\frac{1}{p}, \frac{1}{q})$ is in the interior of the triangle with vertices $(0,0)$, $(1,1)$, and $(\frac{n+1}{n+2}, \frac{n}{n+2})$. 
\end{itemize}
In each case, if $(\frac{1}{p},\frac{1}{q})$ lies outside the closure of the indicated triangle and $q > p$, then no such inequality holds.
\end{theorem}
A consequence of Theorem \ref{modelthm} is that the submanifolds associated to \eqref{modelmaxml}--\eqref{modelharmonic} are model surfaces in the sense of Ricci \cite{ricci1997} and Oberlin \cite{oberlin2008}.

\subsection{Geometrically-formulated results}

Another important sense in which the methods developed here apply broadly is that they do not rely heavily on rigid algebraic properties of the family of submainfolds.
All of the results above will be formulated via the vector field geometry approach as employed by Christ, Nagel, Stein, and Wainger \cite{cnsw1999} and Tao and Wright \cite{tw2003}, among others. Specifically, this means that one uses duality to study the bilinear versions of the operators \eqref{modelmaxml}--\eqref{modelasym} and then analyzes the geometry of the associated projections. In this general setting,
let $\Lsp$ and $\Rsp$ be real-analytic Riemannian manifolds of dimension $n_L$ and $n_R$, respectively, with corresponding Riemannian measures of smooth density $d \mu_\Lsp$ and $d \mu_\Rsp$.
 In the product space $\Lsp \times \Rsp$, suppose that there is a real-analytic submanifold $\onesp$ of dimension $n_L+n_R-\ell$, and let $\pi_L : \onesp \rightarrow \Lsp$ and $\pi_R : \onesp \rightarrow \Rsp$ be the canonical projections onto the first and second factors of the product space $\Lsp \times \Rsp$, respectively. The differentials $d \pi_L$ and $d \pi_R$ are assumed to be everywhere surjective on $\onesp$ and have kernels $\ker d \pi_L$ and $\ker d \pi_R$ whose intersection at each point is trivial. It is also assumed that $\onesp$ has a measure of smooth density $d \mu_{\mathcal M}$ which is dominated by the Riemannian measure, meaning that the quantity $\mu_\onesp(V_1,\ldots,V_{n_L+n_R-\ell})$ at the point $m \in \onesp$ for any $(n_L+n_R-\ell)$-tuple of vectors tangent to $\onesp$ is bounded by a uniform constant (independent of $m$) times the volume of the parallelepiped generated by $V_1,\ldots,V_{n_L+n_R-\ell}$ in the tangent space of $\Lsp \times \Rsp$ at $m$ as measured by the Riemannian metric on $\Lsp \times \Rsp$.

In this geometric setting, the object of study is  the familiar bilinear form
\begin{equation} B (f,g) :=  \int_{\mathcal M} \left( f \circ \pi_L \right) \left( g \circ \pi_R \right) d \mu_{\onesp} \label{bilinear} \end{equation}
which we initially define for all nonnegative Borel measurable functions 
$f$ and $g$ on $\Lsp$ and $\Rsp$, respectively.  The analysis of \eqref{bilinear} focuses on the vector fields $X_L^1,\ldots,X_L^{d_L}$ and $X_R^1,\ldots,X_R^{d_R}$ which form bases of the kernels of $d \pi_L$ and $d \pi_R$, respectively. It is known by results of Christ, Nagel, Stein, and Wainger \cite{cnsw1999} that the bilinear form \eqref{bilinear} can satisfy nontrivial estimates only when the algebra generated by the vector fields $X_L^i$ and $X_R^j$ and all their iterated Lie brackets spans the tangent space of $\onesp$ at every point (this is the so-called H\"{o}rmander condition). However, for most choices of a triple $(n_L,n_R,\ell)$, it is not known exactly what refinement of this spanning condition gives rise to the largest-possible set of $L^p$-$L^q$ estimates.  In the special case $n_L = n_R=n$ and $\ell = 1$, for example, the Phong-Stein rotational curvature condition fills exactly this role. In particular, when $n_L = n_R = n$ and $\ell =1$, the bilinear form \eqref{bilinear} will satisfy the full range of possible $L^p-L^q$ inequalities exactly when
 for any nonzero $v = (v_1,\ldots,v_{n-1}) \in \R^{n-1}$, there is a nonzero $w  \in \R^{n-1}$ such that the set
\begin{equation} \left\{ X_L^1,\ldots,X_L^{n-1},X_R^1,\ldots,X_R^{n-1}, \left[ v \cdot X_L, w \cdot X_R \right] \right\} \label{psrc} \end{equation}
spans the tangent space of $\onesp$ at every point, where, e.g., $v \cdot X_L := \sum_{i=1}^{n-1} v_i X_L^i$.  This is a substantially stronger criterion than the H\"{o}rmander condition because it does not consider the role of higher commutators and it guarantees that there are, in fact, many different ways to find spanning sets.

In general, it is not at all obvious what the natural analogue is of the condition \eqref{psrc} in higher codimensions. There are, however, two new cases, corresponding to the geometric analogues of \eqref{modelmaxml} and \eqref{modelasym}, which the present approach makes it possible to identify:
\begin{theorem}
Suppose that $\onesp$, $\Lsp$, and $\Rsp$ are as defined above and that $\onesp$ is  seven-dimensional and both $\Lsp$ and $\Rsp$ are five-dimensional. Suppose that $X_L^1$ and $X_L^2$ span $\ker d \pi_L$ and likewise for $X_R^1$ and $X_R^2$. Then suppose also that the following curvature conditions hold:
\begin{itemize}
\item For any linearly independent $v,v' \in \R^2$, there is a $w \in \R^2$ such that
\begin{equation*} \left\{ X_L^1,  X_L^2 ,  X_R^1,  X_R^2,  [ X_L^1, v \cdot X_R],  [X_L^2, v \cdot X_R] ,  [w \cdot X_L, v' \cdot X_R] \right\}
\end{equation*}
spans the tangent space of $\onesp$ at every point.
\item For any linearly independent $u,u' \in \R^2$, there is a $w \in \R^2$ such that
\begin{equation*} \left \{ X_L^1,  X_L^2,  X_R^1 , X_R^2,  [u \cdot X_L, X_R^1] ,  [u \cdot X_L, X_R^{2}] ,  [u' \cdot X_L, w \cdot X_R] \right\}
\end{equation*}
spans the tangent space of $\onesp$ at every point.
\end{itemize}
(Note that these conditions are independent of the choice of bases.)
Then for any compact $\Omega \subset \onesp$, we have that \label{variable5}
\[ \left| \int_{\Omega} f( \pi_L(m)) g(\pi_R(m)) d \mu_\onesp(m) \right| \leq C ||f||_{L^{q_L}(\Lsp)} ||g||_{L^{q_R}(\Rsp)} \]
for some $C<\infty$ and all $f \in L^{q_L}$ and $g \in L^{q_R}$ provided that $(\frac{1}{q_L}, \frac{1}{q_R})$ belongs to the interior of the triangle with vertices $(1,0), (0,1)$, and $(\frac{5}{8}, \frac{5}{8})$. Furthermore, no such estimate can hold when $(\frac{1}{q_L},\frac{1}{q_R})$ lies outside the closure of this triangle  when $\frac{1}{q_L} + \frac{1}{q_R} \geq 1$.
\end{theorem}
\begin{theorem}
Let $\mathcal M$, $\Lsp$, and $\Rsp$ as identified at the beginning of this section have the properties that $\onesp$ is $(2d_R+1)$-dimensional and that $\Lsp$ and $\Rsp$ are $2d_R$-dimensional and $(d_R+1)$-dimensional, respectively. Let $X_L$ be a nonvanishing vector field in the kernel of $d \pi_L$, and let $X_R^1,\ldots,X_R^{d_R}$ be vector fields spanning $\ker d \pi_R$. Suppose that
\[ \left\{ X_L ,  X_R^1 , \ldots,  X_R^{d_R} , [X_L, X_R^1] , \ldots,  [X_L, X_R^{d_R}]  \right\}  \]
spans the tangent space at every point of $\onesp$. (Note once again that the condition is independent of the choice of basis.)
Then for any compact $\Omega \subset \onesp$, \label{variableasym}
\[ \left| \int_{\Omega} f( \pi_L(m)) g(\pi_R(m)) d \mu_\onesp(m) \right| \leq C ||f||_{L^{q_L}(\Lsp)} ||g||_{L^{q_R}(\Rsp)} \]
for some $C < \infty$ and all $f \in L^{q_L}$ and $g \in L^{q_R}$ provided that $(\frac{1}{q_L}, \frac{1}{q_R})$ belongs to the interior of the triangle with vertices $(1,0),$  $(0,1)$, and $(\frac{d_R+1}{d_R+2}, \frac{2}{d_R+2})$. No such estimates can hold outside the closure of the triangle when $\frac{1}{q_L} + \frac{1}{q_R} \geq 1$.
\end{theorem}

\subsection{Background}

Interest in $L^p-L^q$ mapping properties of geometric averaging operators can be found in the literature as early as the 1970s in work of Strichartz \cite{strichartz1970} and Littman \cite{littman1971} on the wave equation. As work progressed on a number of different problems, the first major unification of the field came through the work of Phong and Stein \cites{ps1986I,ps1986II,ps1991}, who identified the so-called Phong-Stein rotational curvature condition. Rotational curvature is nonvanishing, for example, when averaging over a translation-invariant family of hypersurfaces with nonvanishing Gaussian curvature. At the time, Phong and Stein were interested in geometric generalizations of the classical Calder\'{o}n-Zygmund singular integral theory. 
This work progressed through the efforts of various authors to reach significant milestones in the work of Christ, Nagel, Stein, and Wainger \cite{cnsw1999} and the more recent advances of Street and Stein \cites{ss2011,ss2012,ss2013} and Street \cite{street2012}. In the breakthrough paper \cite{cnsw1999}, the authors developed a sharp qualitative geometric nondegeneracy condition which sufficed for the study of singular integrals, but in this area of the literature, the quantitative relationship between more refined nondegeneracy criteria and quantitative mapping properties of nonsingular operators is not a primary concern. 

Outside the context of singular integral theory, on the other hand, there is strong and sustained interest in understanding sharp mapping properties of nonsingular geometric averaging operators. The literature in this direction is truly vast. Most of this work focuses on questions relating to averages over curves or hypersurfaces; of the many truly significant contributions in this area, some of the most interesting and noteworthy advances include the work of D. Oberlin \cites{oberlin1987,oberlin1997,oberlin1999,oberlin2002,oberlin2000II}, Iosevich and Sawyer \cite{is1996}, Seeger \cite{seeger1998}, Choi \cites{choi1999,choi2003}, Christ \cite{christ1998}, Greenleaf, Seeger, and Wainger \cite{gsw1999}, Secco \cite{secco1999}, Bak \cite{bak2000}, Tao and Wright \cite{tw2003},  Lee \cite{lee2004}, Dendrinos, Laghi, and Wright \cite{dlw2009},  Erdo{\u{g}}an and R. Oberlin \cite{eo2010}, and Stovall \cites{stovall2010,stovall2014}. See also \cites{gressman2004,gressman2009,gressman2013} for related work on curves and hypersurfaces.

Of the important results briefly noted above, the work upon which the present analysis is most directly built is that of Christ \cite{christ1998}. Christ's paper studies $L^p-L^q$ averages for the convolution operator on $\R^d$ given by
\begin{equation} T f(x) := \int_{0}^1 f( x + (t,t^2,\ldots,t^d)) dt \label{christ} \end{equation}
Christ establishes sharp boundedness (up to for endpoints) via a method he developed which is sometimes referred to as the method of inflation or the method of refinements. A particularly useful feature of this method is that it is very concrete and combinatorial in nature (in contrast to many earlier works which study these operators indirectly via oscillatory integrals or analytic interpolation theorems). Roughly speaking, the idea is to study the $d$-fold alternating composition of $T$ and $T^*$. The map
\begin{equation} \Phi(t_1,\ldots,t_d) := \left( \sum_{i=1}^d (-1)^i t_i, \ldots, \sum_{i=1}^d (-1)^i t_i^d \right) \label{christmap} \end{equation}
which appears in the $d$-fold composition is regarded as a singular change of variables, and the goal of the analysis is, roughly, to bound for any Borel sets $E$ and $F$ in $\R^d$ the size of the set in which $x \in F$, $x + \Phi(t) \in E$, and $t$ is near the set where the Jacobian determinant of $\Phi$ vanishes. In Christ's original paper, it proved useful to over-iterate $T$ and $T^*$ in higher dimensions. This was not due to a failure of the heuristic, but rather it made it possible find a more accommodating $d$-tuples $(t_{i_1},\ldots,t_{i_d})$ to work with in certain degenerate situations. Tao and Wright \cite{tw2003} ultimately merged Christ's approach with the geometric vector field formulation used by Christ, Nagel, Stein, and Wainger \cite{cnsw1999} and others. In so doing, they eliminated the necessity of over-iterating the maps and were able to provide a remarkable and essentially complete calculus for determining $L^p-L^q$ boundedness of averages over curves.

It should also be noted that, aside from averages over curves and hypersurfaces, interest in the Kakeya problem contributed to the development of the multilinear theory of singular geometric averages relating to \eqref{bilinear}. Here the goal is to develop nonlinear generalizations of a series of inequalities found in the literature typically bearing some subset of the names H\"{o}lder, Brascamp, Lieb, Luttinger, Loomis, and Whitney. Some recent work in the harmonic analysis community in this direction includes
Bennett, Carbery, and Wright \cite{bcw2005}, Bennett, Carbery, Christ, and Tao \cite{bcct2008}, Bennett and Bez \cite{bb2010}, and Bennett, Bez, and Guti{\'e}rrez, \cite{bbg2013}. Geometrically, most of this work is characterized by a transversality condition playing the central role that is occupied by rotational curvature in the earlier work on non-multilinear averages (although Tao, Vargas, and Vega \cite{tvv1998}, Stovall \cite{stovall2011}, and Grafakos, Greenleaf, Iosevich, and Palsson \cite{ggip2015} stand as some of the most notable examples of multilinear geometric averages in which both curvature and transversality play important roles).

In contrast to the cases above, very little work has been done to understand linear geometric averaging operators in the case of submanifolds which are neither curves nor hypersurfaces. Various examples of such objects have been identified by 
Ricci \cite{ricci1997} and D. Oberlin \cite{oberlin2008}, and somewhat broader classes were considered by Drury and Guo \cite{dg1991} and Gressman \cite{gressman2015} when the averages were taken over submanifolds of half the ambient dimension, but until now there does not appear to have been any results of a broadly applicable nature analogous to Phong-Stein rotational curvature results or the Tao-Wright result. 

%The present paper provides the first examples of a successful sharp (up to endpoints) adaptation of the geometric vector field framework outside of curves and mutlilinear operators. The challenge here is that the correct curvature conditions are dimensionally-dependent and vary significantly from the Phong-Stein condition or typical transversality conditions. 

\subsection{Approach and organization}

Inflation arguments typically involve the construction of some geometrically-inspired mapping analogous to \eqref{christmap} which can be regarded as a singular coordinate system on one of the spaces involved (which can happen only when the number of variables involved satisfy some favorable factorization constraints). In the case of the arguments that follow, the relevant geometric mappings are typically overdetermined, meaning that the number of parameters exceeds the dimension of the space on which the mapping is built. The sort of over-iteration encountered here is of a fundamentally different nature than the kind encountered, for example, in Christ's work on the corkscrew curve, and the difference leads to a number of new challenges. Chief among them is that the solution-counting problems that one typically encounters before applying the generalized change-of-variables formula are replaced with a much more subtle problem of bounding integrals over very poorly-understood and potentially singular submanifolds which solve some complicated system of equations. For example, an important technical issue is to show that, when $\Phi$ is a sufficiently regular map from some bounded open set in $\R^n$ into $\R^{n-k}$, for any Euclidean ball $B_r(x)$ of radius $r$, the $k$-dimensional Hausdorff measure of the set
\[ \set{ y \in B_r(x)}{ \Phi(y) = c } \]
is (generically in $c$) bounded by some fixed constant times $r^k$. In the particular context in which we would like to apply this result, it is not possible to assume that the Jacobian of $\Phi$ is nonsingular. Consequently, it is necessary to assume additional regularity beyond $C^\infty$. Real analyticity is sufficient (but note that other, larger function spaces would also suffice thanks to the theory of $o$-minimal structures), and in Section \ref{solsec} we prove the necessary regularity results. This can be regarded as a replacement tool for B\'{e}zout's theorem as it is typically applied in inflation arguments and is a largely stand-alone argument.

Technical issues aside, the proof of each part of Theorems \ref{modelthm}--\ref{variableasym} proceeds by reducing the problem to the study of a geometrically-defined sublevel set operator. The path from Radon-like transform to the corresponding sublevel set operator is somewhat lengthy, but for the model case of bilinear-type Radon-like transforms, the result of this calculation may be succinctly stated as follows:
\begin{theorem}
Let $Q : \R^{d_L} \times \R^{d_R} \rightarrow \R^{\ell}$ be any bilinear map with $d_L + d_R \geq \ell$. For any compact set $\Omega \subset \R^{d_L} \times \R^{d_R} \times \R^{\ell}$, let \label{bilinearthm}
\[ B_{Q} (f,g) := \int_{\Omega} f( y, z + Q(x,y) ) g(x,z) dx dy dz. \]
Let $\mathrm{Vol}_Q(x,y)$ equal the largest possible volume of a parallelepiped in $\R^\ell$ generated by vectors in the collection $\{ Q( e_i, y) \}_{i=1}^{d_L} \cup \{ Q(x,e_j) \}_{j=1}^{d_R}$ (where $e_i$'s and $e_j$'s are standard basis vectors) and consider the function
\[ \Phi_Q(x,y) := \frac{\mathrm{Vol}(x,y )}{ ( |x|^2 + |y|^2)^{\frac{d_L + d_R - \ell}{2}}}  \]
together with the associated sublevel set operator
\[ W_{Q,\alpha} ( g, f) := \int_{\Phi_Q(x,y) \leq \alpha} g(x) f(y) dx d y. \]
If there exists $s>0$ and $p_l,p_r \in [1,\infty]$ such that
\begin{equation} W_{Q,\alpha} (\chi_{E^l}, \chi_{E^r}) \lesssim \alpha^s |E^l|^{\frac{1}{p_l}} |E^r|^{\frac{1}{p_r}} \label{sublevelhyp} \end{equation}
for all measurable sets $E^l \subset \R^{d_L}$ and $E^r \subset \R^{d_R}$ and all $\alpha > 0$ (where the notation $\lesssim$ used here and throughout means that the inequality holds up to an implied constant factor which is independent of all varying quantities like $\alpha$ and the sets $E^l$ and $E^r$), then 
\[ B_Q (f,g) \lesssim ||f||_{L^{q_L}} ||g||_{L^{q_R}} \]
for all measurable functions $f$ and $g$ whenever $(q_L^{-1}, q_R^{-1})$ belongs to the interior of the triangle with vertices $(1,0), (0,1)$, and 
 \[ \left( \frac{2 + (s p_l')^{-1}}{3 + (s p_l')^{-1} + (s p_r')^{-1}},  \frac{2 + (s p_r')^{-1}}{3 + (s p_l')^{-1} + (s p_r')^{-1}} \right). \]
If \eqref{sublevelhyp} holds only for sets $E^l$ and $E^r$ belonging to fixed neighborhoods of the origins in $\R^{d_L}$ and $\R^{d_R}$, then the conclusion remains true provided that the diameter of $\Omega$ is sufficiently small. 
\end{theorem}
Establishing boundedness of the sublevel set operator \eqref{sublevelhyp} is not always a simple matter, but its geometric nature makes it amenable to methods not unlike those employed in \cite{gressman2010} to establish boundedness of multilinear determinant-type sublevel set functionals. In Section \ref{sublevellemmasec} we prove useful lemma for establishing such geometric sublevel set functional inequalities which effectively reduces the problem to the estimation of scalar sublevel sets (i.e., not involving integration against any $L^p$ functions).

The general organization of the rest of this paper is as follows:  Section \ref{ttstartsec} establishes a self-contained version of the method of $TT^*T$ which implies $L^p-L^q$ estimates even when $p,q \neq 2$. The proof relies on a trivial application of Christ's method of refinements (interestingly, no further applications of the method are necessary aside from the use of the generalized $TT^*T$ inequality established by Lemma \ref{ttstartlem}). The introduction ends with Section \ref{knappsec}, which gives the Knapp examples establishing sharpness of Theorems \ref{modelthm}--\ref{variableasym}. Section \ref{geometrybg} contains definitions and basic calculations which connect the vector field geometry formulation to $TT^*T$. In particular, this section identifies the geometry of the incidence manifold $\threesp$ which underlies the geometry of $TT^*T$. Section \ref{reductionsec0} uses the geometric calculations of the previous section to reduce the estimation of $T T^* T$ to the problem of bounding the associated geometric sublevel set operator \eqref{sublevels}. All the calculations in Sections \ref{geometrybg} and \ref{reductionsec0} apply in a general way without any nondegeneracy assumptions. Section \ref{calculationsec} establishes a general principle for proving bounds for sublevel set operators and then breaks into subsections to analyze the particular details of the cases put forward in Theorems \ref{modelthm}--\ref{bilinearthm}.
Finally, Section \ref{solsec} establishes the necessary regularity of measures used in Section \ref{reductionsec0}.

\subsection{Generalized $TT^*T$}
\label{ttstartsec}
Suppose $T$ is any positive operator for which one wishes to establish the restricted weak-type inequality of the form 
\[\int_G T \chi_F  \leq C |F|^\frac{1}{q_L} |G|^\frac{1}{q_R}\]
for all measurable sets $F$ and $G$. If we define $T_{GF}$ to be the operator
\[ T_{GF} f (x) := \chi_G(x) T (f \chi_F)(x), \]
then it would suffice, for example, to show that $T_{FG}$ maps $L^2$ to $L^2$ with an operator norm bounded above by $C |F|^{\frac{1}{q_L} - \frac{1}{2}} |G|^{\frac{1}{q_R}-\frac{1}{2}}$
for all such sets $F$ and $G$, since
\[ \int_G T \chi_F  = \int_G T_{GF} \chi_F \leq ||T_{GF}||_{2 \rightarrow 2} |F|^{\frac{1}{2}} |G|^{\frac{1}{2}}. \]
A major advantage of shifting focus to $L^2$ is that the $L^2 \rightarrow L^2$ norm of $T_{GF}$ can be studied via the $L^2 \rightarrow L^2$ norm of $T_{GF} T^*_{GF} T_{GF}$ (or any number of other iterated, alternating compositions of $T_{GF}$ and $T^*_{GF}$). This is the essence of the main argument in this paper. Unfortunately, there are circumstances in which $T_{GF}$ does not behave as well on $L^2$ as one would like:  namely, if one of $q_L$ or $q_R$ exceeds $2$, then for trivial reasons the estimate $||T_{GF}||_{2 \rightarrow 2} \leq C |F|^{\frac{1}{q_L} - \frac{1}{2}} |G|^{\frac{1}{q_R}-\frac{1}{2}}$ cannot hold. This however, turns out to be only a minor inconvenience.
 In the general case, one can use the method of refinements to generalize the inequality $||T_{GF}||_{2 \rightarrow 2}^3 \leq ||T_{GF} T^*_{GF} T_{GF}||_{2 \rightarrow 2}$ to obtain a sufficiently useful substitute for a norm estimate of $||T_{GF}||_{2 \rightarrow 2}$.
\begin{lemma}[Generalized $TT^*T$]
Suppose $T$ is a positive linear operator which maps $L^2 (\Lsp)$ to $L^2(\Rsp)$. \label{ttstartlem}
For any measurable sets $F$ and $G$  in $\Lsp$ and $\Rsp$ with finite, nonzero measure, let
\begin{align}
 F' & := \set{ x \in F}{ T^* \chi_G (x) \geq \frac{\int_G T \chi_F}{3 |F|} } \mbox{ and } \label{fset} \\
  G' & := \set{y \in G}{T \chi_F(y) \geq \frac{\int_G T \chi_F}{3 |G|}}. \label{gset}
  \end{align}
Then
\begin{equation} \left( \frac{1}{3} \int_{G} T \chi_F \right)^3 \leq |F| |G|  \int_{G} T_{GF} T_{G'F'}^* T_{GF} \chi_{F}. \label{refined} \end{equation}
\end{lemma}
\begin{proof}
The proof is a small variation on a familiar argument in the method of refinements. For convenience, define
\[ \delta_F := \frac{1}{3 |F|} \int_{G} T \chi_F \mbox{ and } \delta_G := \frac{1}{3 |G| } \int_G T \chi_F. \]
It follows that
\begin{align*}
 \int_{G}  T_{GF} T_{G'F'}^* T_{GF} \chi_{F}  & = \int_{F'} \left( T^* \chi_G \right) \left( T^*_{G'F'} T_{GF} \chi_F\right) \geq \delta_F \int_{F'} \left( T_{G'F'}^* T_{GF} \chi_F \right) \\
  & = \delta_F \int_{G'} \left( T_{G'F'} \chi_{F'} \right) \left( T \chi_F \right) \geq \delta_F \delta_G \int_{G'} T \chi_{F'}
 \end{align*}
and
\begin{align*}
 \int_{G'} T \chi_{F'} & = \int_{G} T \chi_F - \int_{G \setminus G'} T \chi_{F} - \int_{F \setminus F'} T^* \chi_{G'} \\
& \geq \int_{G} T \chi_F - \delta_G |G| - \delta_F |F| \geq \frac{1}{3} \int_G T \chi_F,
\end{align*}
which together establish \eqref{refined}.
\end{proof}
By \eqref{refined}, one can deduce a restricted weak-type inequality for $T$ if a similar such inequality can be proved for $T_{GF} T^*_{G'F'} T_{GF}$ uniformly in $G$ and $F$. Note that the appearance of $G'$ and $F'$ on the right-hand side of \eqref{refined} represents a very slight gain over the inequality which would be obtained by a simple $T T^* T$ argument. Technically the factor of $\frac{1}{27}$ on the left-hand side of \eqref{refined} is a loss over a direct $TT^*T$ argument, but it is only reasonable to classify this loss as insignificant for the present purposes. The advantage of \eqref{refined}, of course, is that it makes no reference to, and hence no assumptions about, the exponents $q_L$ and $q_R$.

In the context of the averaging operators studied in this paper, the main application of \eqref{refined} involves an argument similar to ``Bourgain's trick,'' in which it will be shown that
\begin{equation}  \int_{G} T_{GF} T_{G'F'}^* T_{GF} \chi_{F}  \leq C \left[  \frac{1}{\alpha} (|F||G|)^{1-\epsilon} + \alpha^s  \int |T^* \chi_G|^{\frac{1}{p_l}} T^*_{G'F'} |T \chi_F|^{\frac{1}{p_r}} \right] \label{btrick} \end{equation}
for some fixed constant $C$ and all positive $\alpha$. The two terms on the right-hand side of \eqref{btrick} follow from estimates of an integral where a certain Jacobian-like quantity is large and small, respectively. In the former case, a generalization of the coarea formula applies, and in the latter case the estimate is reduced to a sublevel set operator estimate. The transition from \eqref{btrick} to an estimate for $T_{GF}T^*_{GF}T_{GF}$, and consequently for $T$ itself, is also fairly immediate.
\begin{lemma}
Suppose that $T$ is a positive operator which is known to be bounded from $L^2(\Lsp)$ to $L^2(\Rsp)$.
Suppose also that there exist exponents $p_l,p_r \in [1,\infty]$, real parameters $\epsilon \in [0,1]$ and $s > 0$, and a finite constant $C$ such that for any measurable sets $F$ and $G$  in $\Lsp$ and $\Rsp$, respectively, and any $\alpha > 0$ the inequality \eqref{btrick} holds, \label{mainlemma}
where $F'$ and $G'$ are the sets defined by \eqref{fset} and \eqref{gset}. Then for some constant $C'$ depending only on $C$ and $s$,
\begin{equation} \int_G T \chi_F \leq C' |F|^{\frac{1}{q_L}} |G|^{\frac{1}{q_R}} \label{rwtype} \end{equation}
where the exponents $q_L$ and $q_R$ are given by
\[ \frac{1}{q_L} := \frac{2 - \epsilon + (s p_l')^{-1}}{3 + (s p_l')^{-1} + (s p_r')^{-1}} \mbox{ and } \frac{1}{q_R} := \frac{2 - \epsilon + (s p_r')^{-1}}{3 + (s p_l')^{-1} + (s p_r')^{-1}}. \]
\end{lemma}
\begin{proof}
By virtue of \eqref{fset} and \eqref{gset},
\begin{align*}
 \int |T^* \chi_G|^{\frac{1}{p_l}} T^*_{G'F'} |T \chi_F|^{\frac{1}{p_r}} &  \leq \delta_F^{-\frac{1}{p'_l}} \int \left( T^* \chi_G \right) T^*_{G'F'} |T \chi_F|^{\frac{1}{p_r}} \\
 & \leq \delta_F^{-\frac{1}{p'_l}} \int \left( T_{G' F'} T^* \chi_G \right)  |T \chi_F|^{\frac{1}{p_r}} \\
&  \leq \delta_F^{-\frac{1}{p'_l}} \delta_G^{-\frac{1}{p_r'}}   \int \left( T_{G' F'} T^* \chi_G \right) ( T \chi_F ) \\
 & =   \delta_F^{-\frac{1}{p'_l}} \delta_G^{-\frac{1}{p_r'}} \int_{G} (T_{GF} T_{G'F'}^* T_{GF} \chi_{F}).
 \end{align*}
 Consequently
 \begin{align*}
 \int_{G} T_{GF} T_{G'F'}^* T_{GF} \chi_{F} & \leq C \left[  \frac{1}{\alpha} (|F||G|)^{1-\epsilon} + \alpha^s  \int |T^* \chi_G|^{\frac{1}{p_l}} T^*_{G'F'} |T \chi_F|^{\frac{1}{p_r}} \right] \\
 & \leq C \left[ \frac{1}{\alpha} (|F||G|)^{1-\epsilon} + \alpha^s \delta_F^{-\frac{1}{p'_l}} \delta_G^{-\frac{1}{p_r'}} \int_{G} T_{GF} T_{G'F'}^* T_{GF} \chi_{F} \right].
 \end{align*}
 To establish \eqref{rwtype}, we may assume that the measures of $F$ and $G$ are both finite, and since $T$ maps $L^2$ to $L^2$ and is positive, we also have that
 \[  \int_{G} T_{GF} T_{G'F'}^* T_{GF} \chi_{F} < \infty. \]
 Choosing $\alpha := (2 C)^{-1/s} \delta_F^{1/(s p_l')} \delta_G^{1/(sp_r')}$ and using  that $\int_{G} T_{GF} T_{G'F'}^* T_{GF} \chi_{F}$ must be finite, it follows that
 \begin{align*}
 \int_G T_{GF} T^*_{G' F'} T_{GF} \chi_F & \leq (2 C)^{ \frac{s+1}{s}} \delta_F^{-\frac{1}{s p'_l}} \delta_G^{-\frac{1}{s p_r'}} (|F||G|)^{1-\epsilon},
 \end{align*}
 so by \eqref{refined}, it follows that
\[ \left( \int_G T \chi_F \right)^{3 + \frac{1}{s p'_l} + \frac{1}{s p'_r}} \leq 27 ( 2 C)^{\frac{s+1}{s}} |F|^{2 - \epsilon + \frac{1}{s p_l'}} |G|^{2 - \epsilon + \frac{1}{s p_r'}}, \]
which gives exactly \eqref{rwtype}. 
\end{proof}

\subsection{Sharpness calculation}
\label{knappsec}

We conclude the introduction with a review of the Knapp examples for \eqref{bilinear}.  Fix any point $m \in \onesp$, and let $F_\epsilon$ be the ball of radius $\epsilon$ centered at $\pi_L(m)$ in $\Lsp$. For sufficiently small $\epsilon$, it must be the case that $\mu_\Lsp(F_\epsilon) \approx \epsilon^{n_L}$. Next, let $M_\epsilon$ equal the set $\pi_L^{-1} F_\epsilon$ restricted to some open set $U \subset \onesp$ with compact closure. Because $d \pi_L$ is surjective, it is possible by the implicit function theorem to find a coordinate system on a neighborhood of $m$ such that, in these coordinates, $\pi_L^{-1} F_\epsilon$ contains the box $[-\epsilon,\epsilon]^{n_L} \times [-1,1]^{n_R - \ell}$ for all $\epsilon$ sufficiently small and is contained in a box of comparable side lengths as well. In particular, then, we must have that $\mu_\onesp(M_\epsilon) \approx \epsilon^{n_L}$.  Lastly, since the kernels of $d \pi_L$ and $d \pi_R$ are transverse, $\pi_L^{-1} \pi_L(m)$ projects via $\pi_R$ to an immersed submanifold of dimension $n_R - \ell$ in $\Rsp$ (which is the same dimension as the embedded submanifold $\pi_L^{-1} \pi_L(m)$ in $\onesp$). Since all points in $M_\epsilon$ are distance at most $\epsilon$ to $\pi_L^{-1} \pi_L(m)$, we must have by smoothness of $\pi_R$ that all points of $\pi_R (M_\epsilon)$ are within distance comparable to $\epsilon$ of the immersed submanifold $\pi_R \pi_L^{-1} \pi_L(m)$. Consequently, if the neighborhood $U$ is reduced to a sufficiently small size, it will be the case that $G_\epsilon := \pi_R(M_\epsilon)$ has $\mu_\Rsp$-measure bounded by a factor times $\epsilon^{\ell}$. 
Consequently if \eqref{bilinear} is bounded for all $f \in L^{q_L}(\Lsp)$ and $g \in L^{q_R}(\Rsp)$, we must have for all $\epsilon$ sufficiently small that
\begin{align*}
 \epsilon^{n_L} \lesssim  |\onesp_\epsilon| & = \int_U \left( \chi_{F_\epsilon} \circ \pi_L \right) \left( \chi_{G_\epsilon} \circ \pi_R \right) d \mu_\onesp \\
 & \lesssim (\mu_\Lsp(F_\epsilon))^{\frac{1}{q_L}} ( \mu_\Rsp(G_\epsilon))^{\frac{1}{q_R}} \lesssim \epsilon^{\frac{n_L}{q_L} + \frac{\ell}{q_R}}.
 \end{align*}
This can only hold when
\begin{equation} \frac{1}{q_L} + \frac{\ell}{n_L} \frac{1}{q_R} \leq 1. \label{knapp1} \end{equation}
By symmetry, we must also have that
\begin{equation} \frac{\ell}{n_R} \frac{1}{q_L} + \frac{1}{q_R} \leq 1. \label{knapp2} \end{equation}
For the constraint \eqref{knapp1}, equality occurs when $(\frac{1}{q_L}, \frac{1}{q_R})$ lies on the line through the points $(1,0)$ and
\begin{equation} \left( \frac{n_R (n_L - \ell)}{n_L n_R - \ell^2} , \frac{n_L( n_R - \ell)}{n_L n_R - \ell^2} \right), \label{mainpt} \end{equation}
and equality occurs in \eqref{knapp2} when $(\frac{1}{q_L}, \frac{1}{q_R})$ lies on the line through the point $(0,1)$ and the point given by \eqref{mainpt}. Quick calculations give that \eqref{mainpt} equals $(\frac{5}{8},\frac{5}{8})$ when $n_L = n_R = 5$ and $\ell = 3$ or $n_L = n_R = 10$ and $\ell = 6$, $(\frac{8}{13},\frac{8}{13})$ when $n_L = n_R = 8$ and $\ell = 5$, and $(\frac{d_R+1}{d_R+2}, \frac{2}{d_R+2})$ when $n_L = 2d_R$, $n_R = d_R+1$, and $\ell = d_R$. This implies sharpness of Theorems \ref{modelthm}--\ref{variableasym} up to cases on the boundary of the respective triangles. Also note for the sake of completeness that Theorem \ref{bilinearthm} will be sharp when $s,p_l$, and $p_r$ satisfy
\[ 1 + \frac{1}{s p_l'} = \frac{\ell}{n_L - \ell} \mbox{ and } 1 + \frac{1}{s p_r'} = \frac{\ell}{n_R - \ell}. \] 
An interesting feature of this criterion is that any estimate for \eqref{sublevelhyp} which satisfies this constraint will continue to do so when it is interpolated with the trivial $L^1 \times L^1$ estimate. Thus, there are always a range of possible estimates for \eqref{sublevelhyp} which would prove best-possible results for Theorem \ref{bilinearthm}.

%The sharp estimate comes when
%\[ \frac{\ell}{n-\ell} = 1 + \frac{1}{s p'}. \]

\section{General geometric framework}

\label{geometrybg}

\subsection{Geometry of one projection}

A {\it $k$-multivector field} on an $n$-dimensional manifold ${\mathcal M}$  is any smooth section of the $k$-th exterior power of tangent bundle of ${\mathcal M}$. These objects are naturally identifiable as dual to $k$-forms (which are instead built on the cotangent bundle) by extending the definition \label{geometry1}
\[ ( d x_{i_1} \wedge \cdots \wedge d x_{i_k}) \left( \frac{\partial}{\partial y_{j_1}} \wedge \cdots \wedge \frac{\partial}{\partial y_{j_k}} \right)  := \det \left[ \! \! \begin{array}{ccc} dx_{i_1} \left( \frac{\partial}{\partial y_{j_1}} \right)  & \! \! \cdots \! \!  &  dx_{i_1} \left( \frac{\partial}{\partial y_{j_k}} \right) \\
\vdots & \! \! \ddots \! \! & \vdots \\
dx_{i_k} \left( \frac{\partial}{\partial y_{j_1}} \right) & \! \! \cdots \! \! &  dx_{i_k} \left( \frac{\partial}{\partial y_{j_k}} \right)
\end{array} \! \! \right]
 \]
by linearity and verifying that the definition is independent of the choice of bases $dx_1,\ldots,dx_n$ of cotangent vectors and $\frac{\partial}{\partial y_1},\ldots,\frac{\partial}{\partial y_n}$ of tangent vectors. In particular, if ${\mathcal M}$ is $n$-dimensional and possesses a nonvanishing $n$-form $\mu_\onesp$, then there is a unique nonvanishing $n$-multivector field $\Xi_\onesp$ such that $\mu_\onesp ( \Xi_\onesp ) = 1$ everywhere (where uniqueness follows from the fact that both $n$-forms and $n$-multivectors form one-dimensional vector spaces at every point).

Now suppose that the $n$-dimensional ${\mathcal M}$ is equipped with a nonvanishing $n$-form $\mu_{\mathcal M}$ and as well as a smooth map $\pi : {\mathcal M} \rightarrow {\mathcal X}$ into some $k$-dimensional manifold ${\mathcal X}$ which is itself equipped with a nonvanishing $k$-form $\mu_{\mathcal X}$. If the differential $d \pi$ is everywhere surjective, then the Implicit Function Theorem guarantees that the fibers of the map $\pi$ are embedded $(n-k)$-dimensional submanifolds of ${\mathcal M}$. It is possible, by the following construction, to identify a unique $(n-k)$-multivector field on ${\mathcal M}$ which encodes the kernel of $d \pi$ at every point. To begin, fix any point $m \in \onesp$ and let $V_1,\ldots,V_{n-k}$ be any linearly-independent vectors in the kernel of $d \pi$ at $m$.  For any additional tangent vectors $X_1,\ldots,X_k$ at  $m$, consider the quantities
\[ \mu_{\onesp} \left( X_1 \wedge \cdots X_{k} \wedge V_1 \wedge \cdots \wedge V_{n-k} \right) \mbox{ and } \mu_{\mathcal X} \left( d \pi (X_1) \wedge \cdots \wedge d \pi (X_k) \right). \]
Both quantities are unchanged if any vector $X_i$ is replaced by $X_i + \sum_{j=1}^{d-k} c_j V_j$ and consequently both expressions extend to alternating $k$-linear forms on the vector space $T_m({\onesp}) / \ker d \pi$, which is $k$-dimensional. Therefore uniqueness of the determinant implies that they differ by a constant independent of the choice of $X_1,\ldots,X_k$. Since $\mu_{\onesp}$ is assumed nonvanishing, we may always choose $V_1,\ldots,V_{n-k} \in \ker d \pi$ so that
\begin{equation} \label{multivec}
\begin{split} \mu_{\onesp} &  \left( X_1 \wedge \cdots X_{k} \wedge V_1 \wedge \cdots \wedge V_{d-k} \right) \\ & = \mu_{\mathcal X} \left( d \pi (X_1) \wedge \cdots \wedge d \pi (X_k) \right) \ \ \forall X_1,\ldots,X_k \in T_m({\onesp}). 
\end{split}
\end{equation}
Notice that $d \pi(X_1) \wedge \cdots \wedge d \pi(X_k)$ depends only on $X_1 \wedge \cdots \wedge  X_k$, and so will be abbreviated $d \pi(X_1 \wedge \cdots \wedge X_k)$ for convenience.
Because the space of $(n-k)$-multivectors generated by $\ker d \pi$ is one-dimensional, the value of $V_1 \wedge \cdots \wedge V_{d-k}$ is constant for any $V_1,\ldots,V_{d-k}$ satisfying \eqref{multivec}. Thus the multivector
\[ {\mathbf V} := V_1 \wedge \cdots \wedge V_{n-k} \]
depends only on $\mu_\onesp$, $\mu_{\mathcal X}$, and the map $\pi$.

Since ${\mathbf V}$ is nonvanishing, when it is restricted to the fibers of $\pi$, it is dual to a unique nonvanishing $(n-k)$-form on those fibers, which we will call $\mu^\pi$ (note that uniqueness here only holds on the fibers; any extension of $\mu^\pi$ to all of ${\onesp}$ will not be unique). Therefore if $V_1,\ldots,V_{n-k}$ are {\it any} vectors in $\ker d \pi$ at the point $m$ (not necessarily normalized as in \eqref{multivec}), it must be the case that
\begin{equation*}
\begin{split} \mu_{\onesp} &  \left( X_1 \wedge \cdots X_{k} \wedge V_1 \wedge \cdots \wedge V_{n-k} \right) \\ & = \mu_{\mathcal X} \left( d \pi ( X_1 \wedge \cdots \wedge X_k) \right) \mu^\pi (V_1 \wedge \cdots \wedge V_{n-k})
\end{split}
\end{equation*}
for all $X_1,\ldots,X_k \in T_m(\onesp)$ and all $V_1,\ldots,V_{n-k} \in \left. \ker d \pi \right|_m$. This equality leads to a geometric Fubini/coarea formula for the integration of functions $f$ against the density $|\mu_{\onesp}|$: we may factor integrals over $\onesp$ into an integral over the fibers of $\pi$ followed by an integral over ${\mathcal X}$:
\begin{equation} \int_{\onesp} f ~ d |\mu_{\onesp}| = \int_{\mathcal X} \left[ \int_{\pi^{-1}(x)} f~ d |\mu^\pi| \right] d\left|\left.\mu_{\mathcal X}\right|_x \right|. \label{coarea0}
\end{equation}
One small but important note is that this construction works equally well if $\onesp$ and $\mathcal X$ are only equipped with smooth nonvanishing densities $|\mu_{\onesp}|$ and $|\mu_{\mathcal X}|$. The only problem introduced by this change is an ambiguity in the sign of $\mathbf V$. This does not affect the integration formula \eqref{coarea0} and uniqueness can be restored by working with what will be called {\it unsigned multivectors}, which are simply multivectors modulo scalar multiplication by $\pm 1$. We also note that a $k$-multivector (signed or unsigned) will be called {\it decomposable} when it may be written as a wedge product of $k$ vectors in $T_{m}(\onesp)$.

\subsection{Geometry of two projections and $T T^*T$}

The intrinsic geometry of \eqref{bilinear} is governed by the structure of the two projections $\pi_L$ and $\pi_R$. \label{geometry2}
For any measurable sets $F$ and $G$ in $\Lsp$ and $\Rsp$, we will be interested in the restriction of \eqref{bilinear} to $F \times G$ defined by
\begin{equation} B_{GF} (f,g) :=  \int_{\onesp} \left( (f \chi_F) \circ \pi_L \right) \left( (g \chi_G) \circ \pi_R \right) d \mu_{\onesp}. \label{ifgdef} \end{equation}
By \eqref{coarea0}, there are measures of smooth density $d \mu^{\pi_L}$ and $d \mu^{\pi_R}$ on the fibers of $\pi_L$ and $\pi_R$, respectively, such that
\begin{align}
 \int_{\onesp} \left( f \circ \pi_L \right) \left( g \circ \pi_R \right) d \mu_{\onesp}  &  = \int_{\Lsp} f(x) \left[ \int_{\pi_{L}^{-1}(x)} g \circ \pi_R ~ d \mu^{\pi_L} \right] d \mu_{\Lsp} (x) \label{integ1} \\
 & 
 = \int_{\Rsp} g(x) \left[ \int_{\pi_{R}^{-1}(x)} f \circ \pi_L ~ d\mu^{\pi_R} \right] d \mu_{\Rsp} (x) \label{integ2}
 \end{align}
 for any measurable functions $f$ on ${\mathcal X}_L$ and $g$ on ${\mathcal X}_R$. In particular, replacing $f$ and $g$ by $f \chi_F$ and $g \chi_G$ shows that the operators
  \begin{align}
  T_{GF} f (x_R) & :=  \chi_G(x_R) \int_{\pi_{R}^{-1}(x_R)} (f \chi_F) \circ \pi_L ~ d\mu^{\pi_R}, \label{tdef} \\
    T_{GF}^* g (x_L) & :=  \chi_F(x_L) \int_{\pi_{L}^{-1}(x_L)} (g \chi_G) \circ \pi_R ~ d\mu^{\pi_L},  \label{tsdef}
  \end{align}
 satisfy
 \begin{equation*} B_{GF}(f,g) = \int_{{\mathcal X}_R}\left(T_{GF} f \right) g d \mu_{\Rsp} = \int_{\Lsp} f \left( T_{GF}^* g \right) d \mu_{\Lsp}. \end{equation*}
 Note that by \eqref{integ1} and \eqref{integ2}, if $\Omega$ is a compact subset of $\onesp$, it must be that 
 \[ \left| \int_{\onesp \cap \Omega} (f \circ \pi_L) (g \circ \pi_R) \right| \leq C \min\{ ||f||_{L^1(\Lsp)} ||g||_{L^\infty(\Rsp)}, ||f||_{L^\infty(\Lsp)} ||g||_{L^1(\Rsp)} \} \]
 for all $f$ and $g$ because the integrals
 \[ \int_{\pi_L^{-1}(x) \cap \Omega} d \mu^{\pi_L} \mbox{ and } \int_{\pi_R^{-1}(x) \cap \Omega} d \mu^{\pi^L} \]
 will be finite and uniformly bounded as a function of $x \in \pi_L(\Omega)$ and $\pi_R(\Omega)$ respectively (which will both be compact as well). Boundedness of $T_{GF}$ from $L^2$ to $L^2$ (the main technical hypothesis of Lemma \ref{mainlemma}) when restricted to integration over $\Omega$ must follow by the Schur test.

 Returning to \eqref{tdef} and \eqref{tsdef}, these operators will be analyzed via the generalized $T T^* T$ inequality \eqref{refined}, which requires study of the more elaborate object
 \begin{align*}
 \int_{\Rsp} (T_{GF} T_{G'F'}^* T_{Gf} f) g d \mu_{\Rsp} & = \int_{\Lsp} ( T_{G'F'}^* T_{GF} f) (T_{GF}^* g) d \mu_{\Lsp}   = B_{G'F'} ( T_{GF}^* g, T_{GF} f). 
 \end{align*}
In terms of integration, this last object may be expressed in terms of an integral with respect to some measure $d \mu$ of smooth, nonvanishing density on the space
\begin{equation}
\begin{split}
 \threesp := \left\{ (m^l,m^c,m^r) \in \onesp \right. & \times \onesp \times \onesp 
 \\
 & \left| \ \vphantom{\onesp} \pi_L(m^l) = \pi_L(m^c) \mbox{ and } \pi_R(m^r) = \pi_R(m^c)\right\}.
 \end{split} \label{imanifold} \end{equation}
For any $p := (m^l,m^c,m^r) \in \threesp$, let $\pi^j(p) := m^j$ for any superscript $j=l,c,r$.  Likewise, define $\pi^{i}_j := \pi_j \circ \pi^i$ for any $i=L,R$ and any $j=l,c,r$. One may expand $B_{G'F'} ( T_{GF}^* g, T_{GF} f)$ using \eqref{ifgdef} to write $B_{G'F'}$ as an integral over $\mathcal M$ and then write $T_{GF}^* g$ and $T_{GF} f$ in terms of \eqref{tsdef} and \eqref{tdef}, respectively, to conclude that $B_{G'F'} ( T_{GF}^* g, T_{G'F'} f)$ equals
\begin{equation}
B^{(3)}_{G'F'}(g,f) :=  \int_{\threesp} \! \! \left((g \chi_G) \circ \pi_R^l \right) \left( \chi_{F'} \circ \pi_L^c \right) \left(  \chi_{G'} \circ \pi_R^c \right) \left( (f \chi_F) \circ \pi_L^r \right)  d \mu \label{quadlinear}
 \end{equation}
 when $d \mu$ is simply taken to be the measure generated by the integrals over ${\mathcal M}$ and the fibers of $\pi_L$ and $\pi_R$.

\subsection{Construction of vectors tangent to the incidence manifold $\threesp$}

Any tangent vector $Z \in T_p( \threesp)$ is uniquely determined by the triple \label{geometry2c}
\[ \left. \ang{d \pi^l(Z), d \pi^c(Z), d \pi^r(Z))} \right|_p \in T_{\pi^l(p)}({\mathcal M}) \times T_{\pi^c(p)}({\mathcal M}) \times T_{\pi^r(p)}({\mathcal M}), \]
and any such triple identifies a tangent vector exactly when it satisfies the compatibility conditions $d \pi^l_L(Z) = d \pi^c_L(Z)$ and $d \pi^r_R(Z) = d \pi^c_R(Z)$ (here and throughout the rest of the paper, brackets will be used to represent elements of a Cartesian product of vector spaces). In other words, if one wishes to find $U, V$, and $W$ so that
\[ \left. \ang{ U,V,W } \right|_p  \in T_{\pi^l(p)}({\mathcal M}) \times T_{\pi^c(p)}({\mathcal M}) \times T_{\pi^r(p)}({\mathcal M}) \]
is tangent to $\threesp$, one only needs to verify that
\begin{equation}  d \pi_L(U) = d \pi_L(V) \mbox{ and } d \pi_R(V) = d \pi_R(W). \label{constr} \end{equation}
Since both $d \pi_L$ and $d \pi_R$ are everywhere surjective, given any one of $U,V$, or $W$ it is always possible to solve \eqref{constr} for the other two, but the solution is never unique. Note, for example, that when $\pi^l(p) = \pi^c(p)$, $U = V$ is always possible; likewise when $\pi^r(p) = \pi^c(p)$, $V = W$ is always possible. To consistently choose solutions, let us first fix at each point $p \in \threesp$  a map 
\[ \inc_p : T_{\pi^c(p)}(\onesp) / (\ker d \pi_L + \ker d \pi_R) \rightarrow T_p(\threesp) \]
 such that for every $v \in T_{\pi^c(p)}(\onesp) / (\ker d \pi_L + \ker d \pi_R)$ we have
 \begin{equation} v = d \pi^c \left(\inc_p v \right) \mbox{ modulo } \left. \left(\ker d \pi_L + \ker d \pi_R \right) \right|_{\pi^c(p)}. \label{choices} \end{equation}
 Such a map $\inc_p$ can easily be constructed (for example) by choosing any maximal set of tangent vectors $\{Z_i\}$ of $\threesp$ at $p$ whose ``center parts'' $d \pi^c(Z_i)$ are linearly independent modulo $\ker d \pi_L + \ker d \pi_R$ and defining $\inc_p$ to send $d \pi^c(Z_i)$ (modulo the sum of kernels) to $Z_i$ for each $i$.
 To see how $\inc_p$ can be used to consistently construct tangent vectors, let us show that for any $U \in T_{\pi^l(p)}({\mathcal M})$ there must exist unique choices of $v \in T_{\pi^c(p)}(\onesp) / (\ker d \pi_L + \ker d \pi_R) $ and $V_R \in \ker d \pi_R$ so that 
\begin{equation} \left.  \inc_p^l U \right|_p := \ang{U, d \pi^c \inc_p v + V_R, d \pi^r \inc_p v} \label{hardest1} \end{equation}
 is tangent to $\threesp$ at $p$. Because $d \pi_R^c \inc_p = d \pi_R^r \inc_p$ and $d \pi_R(V_R) = 0$, it suffices to show that $d \pi_L(U) = d \pi_L ( d \pi^c \inc_p v + V_R)$. After counting dimensions, we need only show that $d \pi_L( d \pi^c \inc_p v + V_R) = 0$ implies $v = 0$ and $V_R = 0$. But $d \pi_L( d \pi^c \inc_p v + V_R) = 0$ implies $d \pi^c \inc_p v \in \ker d \pi_L + \ker d \pi_R$, which by \eqref{choices} implies that $v = 0$. Then $v=0$ forces $V_R = 0$ as well because $\ker d \pi_L \cap \ker d \pi_R$ is trivial. By similar reasoning, for any $W \in T_{\pi^r(p)}(\onesp)$, there exist unique $v \in T_{\pi^c(p)}(\onesp) / (\ker d \pi_L + \ker d \pi_R)$ and $V_L \in \ker d \pi_L$ so that
 \begin{equation} \left. \inc_p^r  W \right|_p := \ang{ d \pi^l \inc_p v, d \pi^c \inc_p v + V_L, W} \in T_p(\onesp). \label{hardest2} \end{equation}
 Two important special cases of these constructions include the vectors
% 
%  Finally, for convenience, we refer back to \eqref{leftpull} and \eqref{rightpull} relating vectors along the fibers of $\pi_L$ and $\pi_R$, respectively, and make the definitions
%\begin{align*}
% Q^l (  \left. V \right|_{\pi^l(p)} )  & := Q_L^{\pi^c(p) \pi^l(p)} (V) , \\
% Q^r( \left. V \right|_{\pi^r(p)} ) &  := Q_R^{\pi^c(p) \pi^r(p)}(V).
% \end{align*}
%
%With these auxiliary definitions in place, we now define some special vector fields on $\threesp$ which play particularly important roles in the calculations to come.
%We begin by fixing any vectors $X_L$ and $X_R$ belonging to the kernel of $d \pi_L$ at $\pi^l(p)$ and the kernel of $d \pi_R$ at $\pi^r(p)$, respectively.  The triples 
\begin{align}
\left. \inc_p^l X_L \right|_p  :=  \left. \ang{ X_L , 0 , 0} \right|_p \mbox{ and } 
 \left. \inc_p^r X_R \right|_p  := \left. \ang{ 0, 0,  X_R} \right|_p, \label{simplests} 
\end{align}
when $X_L$ belongs to $\ker d \pi_L$ at $\pi^l(p)$ and when $X_R$ belongs to $\ker d \pi_R$ at $\pi^r(p)$.  Another very important set of linear maps to consider in this direction have the forms
\begin{align}
 {\mathcal C}^l_p : T_{\pi^l(p)}(\onesp) \rightarrow T_{\pi^c(p)}(\onesp) / (\ker d \pi_L + \ker d \pi_R),  \label{lproj} \\
  {\mathcal C}^r_p : T_{\pi^r(p)}(\onesp) \rightarrow T_{\pi^c(p)} (\onesp) / (\ker d \pi_L + \ker d \pi_R), \label{rproj}
\end{align}
and are defined so that ${\mathcal C}^l_p U$ is the equivalence class of $d \pi^c \inc_p^l U$ and ${\mathcal C}^r_p W$ is the equivalence class of $d \pi^c \inc_p^r W$ (modulo $\ker d \pi_L + \ker d \pi_R$). It so happens that these maps ${\mathcal C}^l_p $ and ${\mathcal C}^r_p$ are independent of the choice of $\inc_p$, since, for example,
\[ \ang{ U , d \pi^c \inc_p v + V_R, d \pi^r \inc_p v} - \ang{U,d \pi^c \inc'_p v' + V_R', d \pi^r \inc'_p v'} \]
is tangent to $\threesp$, so by \eqref{constr} it must be the case that
\[ 0 = d \pi_L ( (d \pi^c \inc_p v - d \pi^c \inc_p' v') + V_R - V_R'), \]
which is to say that $d \pi^c \inc_p v - d \pi^c \inc_p' v'$ vanishes modulo $\ker d \pi_L + \ker d \pi_R$. A satisfying consequence of this observation is that
in formulas \eqref{hardest1} and \eqref{hardest2}, we have $v = {\mathcal C}_p^l U$ and $v = {\mathcal C}_p^r W$, respectively, since, in the case of ${\mathcal C}_p^l U$, $d \pi^c \inc_p v$ is equivalent to $v$ modulo $\ker d \pi_L + \ker d \pi_R$ by definition of $\inc_p$ and, also modulo the sum of kernels, is equal to ${\mathcal C}_p^l U$ by definition.

This section concludes with a calculation demonstrating how the maps ${\mathcal C}_p^l$ and ${\mathcal C}_p^r$ encode the Lie algebra generated by the vector fields $X_L$ and $X_R$. This will be an important piece of the variable coefficient Theorems \ref{variable5} and \ref{variableasym}.
\begin{lemma}
Let ${\mathcal C}_p^l$ and ${\mathcal C}_p^r$ be the maps \eqref{lproj} and \eqref{rproj} which map vectors at $\pi^l(p)$ and $\pi^r(p)$, respectively, to vectors at $\pi^c(p)$ modulo the kernels $d \pi_L$ and $d \pi_R$
If $V$ is any smooth vector field on $\onesp$, ${\mathcal C}_p^l V$ is a smooth function on the manifold which is vector-valued in a vector space depending on $\pi^c(p)$. Consequently, along the submanifold where $\pi^c(p)$ is constant, ${\mathcal C}_p^l V$ can be intrinsically differentiated (i.e., independently of any choice of basis or coordinates).
Likewise ${\mathcal C}_p^r V$ can be intrinsically differentiated along the submanifold where $\pi^c(p)$ is constant. In particular, if $X_L$ and $X_R$ are any vectors in the kernel of $d \pi_L$ and $d \pi_R$, respectively, it must be the case that
\begin{align}
 \left. X_L \right|_{\pi^l(p)} {\mathcal C}_p^l V & = {\mathcal C}_p^l [X_L,V], \label{leftcurv} \\
 \left. X_R \right|_{\pi^r(p)} {\mathcal C}_p^r V & = {\mathcal C}_p^r [X_R,V]. \label{rightcurv}
 \end{align}
 \end{lemma}
\begin{proof}
We will give the calculation for ${\mathcal C}_p^l$ only, as the calculation for ${\mathcal C}_p^r$ is completely symmetric. Furthermore, since $\pi^c(p)$ is constant as $\pi^l(p)$ varies in the direction $X_L$, any vector-valued function with values in $\left. \ker d\pi_L + \ker d \pi_R \right|_{\pi^c(p)}$ will have derivatives of all orders belonging to that same vector subspace. Therefore it suffices to show that
\begin{equation} \left. X_L \right|_{m^l} \left( \left. (d \pi^c \inc_p^l V) \right|_{m^c} \right) = \left. (d \pi^c \inc_p^l [X_L,V]) \right|_{m^c} + X_R''' + X_L'\label{derivsuf} \end{equation}
where $X_R''' \in \ker d \pi_R$ at $m^c$ and $X_L' \in \ker d \pi_L$ at $m^c$.

Given a smooth vector field $V$, the formula \eqref{hardest1} together with the constraints \eqref{constr} implies the existence of a smooth vector function $X_R'$ with values in $\ker d \pi_R$ at $\pi^c(p)$ such that
\[ d \pi_L(\left. V \right|_{\pi^l}(p)) = d \pi_L( d \pi^c \inc_p^l V + X_R'). \]
Let us denote the point $\pi^l(p)$ by $m^l$ and regard $\pi^c(p)$ and $\pi^r(p)$ as constants. Then for any smooth function $f$ on $\Lsp$,
\begin{align*}
 \left. X_L \right|_{m^l} \left( \left. V \right|_{m^l} \left. f \circ \pi_L  \right|_{m^l} \right) & = \left. [ X_L, V ] \right|_{m^l} \left. f \circ \pi_L \right|_{m^l}
 \end{align*}
 since $X_L$ belongs to the kernel of $d \pi_L$.  However,
 \begin{align*} \left. V \right|_{m^l} \left. f \circ \pi_L \right|_{m^l} &  = \left.(d \pi_L V)\right|_{\pi_L(m^l)} \left. f \right|_{\pi_L(m^l)} 
 \\ & = \left. d \pi_L ( d \pi^c \inc_p^l V + X_R')\right|_{\pi_L(m^l)} \left. f \right|_{\pi_L(m^l)} \\
 & = \left. d \pi_L ( d \pi^c \inc_p^l V + X_R')\right|_{\pi_L(m^c)} \left. f \right|_{\pi_L(m^c)} \\
 & = \left. ( d \pi^c \inc_p^l V + X_R')\right|_{m^c} \left. f \circ \pi_L \right|_{m^c},
 \end{align*}
 so it must be the case that
\[ \left. X_L \right|_{m^l} \left( \left. ( d \pi^c \inc_p^l V + X_R')\right|_{m^c} \left. f \circ \pi_L \right|_{m^c} \right) = \left. [ X_L, V ] \right|_{m^l} \left. f \circ \pi_L \right|_{m^l}. \]
Now the only dependence of the function $\left. ( d \pi^c \inc_p^l V + X_R')\right|_{m^c} \left. f \circ \pi_L \right|_{m^c}$ on $m^l$ is through the vector-valued function $\left. ( d \pi^c \inc_p^l V + X_R')\right|_{m^c}$ itself.
Therefore we must have that
\[ \left[ \left. X_L \right|_{m^l} \left( \left. ( d \pi^c \inc_p^l V + X_R')\right|_{m^c}  \right) \right] \left. f \circ \pi_L \right|_{m^c} = \left. [ X_L, V ] \right|_{m^l} \left. f \circ \pi_L \right|_{m^l},\]
which is to say that
\begin{align*}
 d \pi_L \left[ \left. X_L \right|_{m^l} \left( \left. ( d \pi^c \inc_p^l V + X_R')\right|_{m^c}  \right) \right]  & = d \pi_L \left. [ X_L, V ] \right|_{m^l} \\
 & = d \pi_L ( \left. d\pi^c \inc_p^l [X_L,V] \right|_{m^c}  + \left. X_R'' \right|_{m^c})
 \end{align*}
for some smooth vector-valued function $X_R''$ with values in $\ker d \pi_R$. Thus
\begin{align*}
  \left. X_L \right|_{m^l} \left( \left. ( d \pi^c \inc_p^l V + X_R')\right|_{m^c}  \right) 
 & = \left. ( d \pi^c \inc_p^l [X_L,V] + X_R'')\right|_{m^c} + \left. X_L' \right|_{m^c},
 \end{align*}
for some $X_L'$ with values in $\ker d \pi_L$, which implies \eqref{derivsuf} when we set $X_R''' := X_R'' - \left. X_L \right|_{m^l}  X_R'$, since the derivative of a function with values in $\ker d \pi_R$ is also in $\ker d \pi_R$.
\end{proof}

\section{Reduction of geometric averages to sublevel set estimates}

\label{reductionsec0}

Recall that the main goal is to prove the inequality \eqref{btrick}, where the left-hand side is now given by the bilinear form \eqref{quadlinear} acting on the characteristic functions $\chi_G$ and $\chi_F$. The first term on the right-hand side comes from an application of a coarea-type formula, which is essentially the only formula one can appeal to when trying to prove some approximate boundedness of \eqref{quadlinear} on $L^1(\Lsp) \times L^1(\Rsp)$. Unfortunately, such boundedness does not actually hold. However, we may regard the upcoming quantity $\mathcal K$ identified in \eqref{keq} as a Jacobian which governs finiteness of the functional. We will use $\mathcal K$ to break the functional \eqref{quadlinear} into two pieces. On the first piece, we will be able to regard $\mathcal K$ as essentially large (although we will not decompose directly in terms of the value of $\mathcal K$, but rather a slightly more elaborate function depending on it). The second piece will reduce to a sublevel set functional which, if bounded, gives the second term on the right-hand side of \eqref{btrick}.

\subsection{Projection and a coarea-type formula for $T T^* T$} 

In this section we consider the effect in the bilinear functional \eqref{quadlinear} for $TT^*T$ of placing both $g$ and $f$ in $L^1$. Unfortunately, the result is not always finite, but we will calculate the Jacobian-type which governs finiteness. The key is to understand and quantify the degeneracy of the map $\Pi : \threesp \rightarrow \Lsp \times \Rsp$ given by
 \begin{equation} \Pi(p) := (\pi_L^r(p),\pi_R^l(p)). \label{bigpi} \end{equation}
 For convenience, let $d_L$ and $d_R$ be the dimensions of the kernels of $d \pi_L$ and $d \pi_R$, i.e., $d_L := n_R - \ell$ and $d_R := n_L - \ell$.
 Now $\Pi$ maps the incidence manifold $\threesp$, which is a space of dimension $2 d_L + 2 d_R + \ell$ into $\Lsp \times \Rsp$, which has dimension $n_L + n_R$. Consequently we expect the fibers of $\Pi$ to have dimension $\kappa := d_L + d_R - \ell$ (assuming that $\kappa \geq 0$) and
 by the coarea formula, we further expect
 \[ \int_{\threesp} F d \mu = \int_{\Lsp \times \Rsp} \left[ \int_{\Pi^{-1}(x_L,x_R)} F \frac{d {\mathcal H}^{\kappa}}{ J} \right] d \mu_\Lsp(x_L) d \mu_\Rsp(x_R) \]
 where $d {\mathcal H}^\kappa$ is the $\kappa$-dimensional Hausdorff measure and $J$ is the corresponding Jacobian. Technical justification aside, this turns out not to be a particularly convenient way to express the integral of $F$ on $\threesp$, since the Jacobian $J$ is somewhat difficult to analyze. For this reason, we will derive a slightly different expression for the integral which, among other things, has an explicit dependence on the maps which encode the generalized rotational curvature.

To that end, fix any vectors $X_L^i$ and $X_R^j$ at $\pi^l(p)$ belonging to the kernels of $d \pi_L$ for all $i=1,\ldots,d_L$ and $d \pi_R$ for all $j = 1,\ldots,d_R$, respectively. We have by \eqref{simplests} and \eqref{hardest1} that
\begin{equation} d \Pi (\inc_p^l X_L^{i}) = \ang{ 0, d \pi_R (X_L^i)} \mbox{ and } d \Pi(\inc_p^l X_R^{j}) = \ang{d \pi^r_L \inc_p^l X_R^{j}, 0} \label{calc1} \end{equation}
%We also have that
%\[ d \Pi ( \lr{\imath} X_R^{j} -  \imath^l X_R^j) = \ang{0 , - d \pi^l_R \imath^l X_R^j }. \] 
since $d \pi^r \inc_p^l X_L^i = 0$ and $d \pi^l_R \inc_p^l X_R^j = d \pi_R X_R^j = 0$.
In the calculations below, bold symbols are used to represent unsigned decomposable multivectors; when nonbold, enumerated variables have also been defined, bold will represent the ordered wedge product of the enumerated vectors. For example,
\begin{equation}
{\mathbf X}_L := X_{L}^{1} \wedge \cdots \wedge X_{L}^{d_L} \mbox{ and } {\mathbf X}_R := X_R^{1} \wedge \cdots \wedge X_R^{d_R}. \label{multi1} 
\end{equation}
We will assume that the vectors $X_{L}^i$ and $X_R^j$  satisfy the normalization condition
\begin{equation} \mu_{\onesp} ({\mathbf X}_L \wedge {\mathbf \Xi}) = \mu_\Lsp( d \pi_L {\mathbf \Xi}) \mbox{ and }  \mu_{\onesp} ( {\mathbf X}_R \wedge {\mathbf \Xi}) = \mu_\Rsp( d \pi_R {\mathbf \Xi}) \label{lrnorm} \end{equation}
where ${\mathbf \Xi}$ is any decomposable unsigned $n_L$-multivector  in the former case and  $n_R$-multivector in the latter. Recall that, once normalized in this way,  ${\mathbf X}_L$ and ${\mathbf X}_R$ are uniquely determined.
If $S_L \subset \{1,\ldots,d_L\}$ and $S_R \subset \{1,\ldots,d_R\}$, we will also define
\begin{equation} {\mathbf X}_L^{S_L} := X_L^{i_1} \wedge \cdots \wedge X_L^{i_{\#S_L}} \mbox{ and } {\mathbf X}_{R}^{S_R} := X_R^{j_1} \wedge \cdots \wedge X_R^{j_{\#S_R}} \label{multi2} \end{equation}
where $i_1,\ldots, i_{\#S_L}$ and $j_1,\ldots,j_{\#S_R}$ are enumerations of $S_L$ and $S_R$, respectively.

Next, if $X_L^i$ and $X_R^j$ are also defined at $\pi^r(p)$, still belonging to $\ker d \pi_L$ and $\ker d \pi_R$, respectively, and satisfy the  normalization \eqref{lrnorm}, then we have
\begin{equation} d \Pi (\inc_p^r X_L^{i}) = \ang{ 0, d \pi_R^l \inc_p^r X_L^i } \mbox{ and } d \Pi(\inc_p^r X_R^{j}) = \ang{d \pi_L(X_R^j),0} \label{calc2} \end{equation}
analogously to \eqref{calc1} as well as the formula
\begin{equation} d \Pi (\inc_p^r X_L^{i} - \inc_p {\mathcal C}_p^r X_L^i ) = \ang{ - d \pi_L^r \inc_p {\mathcal C}_p^r X_L^i, 0}. \label{calc3} \end{equation}
 Finally, fix any elements $v^1,\ldots,v^{\ell} \in T_{\pi^c(p)}(\mathcal M)/ (\ker d\pi_L + \ker d \pi_R)$.
 Assuming that $\#S_R + \#S_L = \ell$, by \eqref{calc1}, \eqref{calc2}, and \eqref{calc3}, it must be the case that
 \begin{align*}
 d \Pi & \left(  \inc_p^l {\mathbf X}_R^{S_R}  \wedge \inc_p^r (  {\mathbf X}_L^{S_L} \wedge {\mathbf X}_R ) \wedge ( \inc_p^l {\mathbf X}_L \wedge \inc_p {\mathbf v}) \right) \\
 & =  \ang{ d \pi_L^r \left( \inc_p^l {\mathbf X}_R^{S_R} \wedge  \inc_p {\mathcal C}_p^r {\mathbf X}_L^{S_L} \wedge \inc_p^r {\mathbf X}_R \right) , 0} \wedge \ang{0, d \pi_R^l ( \inc_p^l {\mathbf X}_L \wedge \inc_p {\mathbf v} )} 
 \end{align*}
since we may replace each $\inc_p^r X_L^{i}$ by $\inc_p^r X_L^{i} - \inc_p {\mathcal C}_p^r X_L^i $ so long as the $v^i$ span $T_{\pi^c(p)}(\mathcal M)/ (\ker d\pi_L + \ker d \pi_R)$.
In particular, it follows that
 \begin{align*}
 \mu &_{\Lsp \times \Rsp}  d \Pi  \left(  \inc_p^l {\mathbf X}_R^{S_R}  \wedge \inc_p^r (  {\mathbf X}_L^{S_L} \wedge {\mathbf X}_R ) \wedge ( \inc_p^l {\mathbf X}_L \wedge \inc_p {\mathbf v}) \right) \\
 & =  \mu_{\Lsp} d \pi_L^r \left( \inc_p^l {\mathbf X}_R^{S_R} \wedge  \inc_p {\mathcal C}_p^r {\mathbf X}_L^{S_L} \wedge \inc_p^r {\mathbf X}_R \right) \mu_{\Rsp} d \pi_R^l ( \inc_p^l {\mathbf X}_L \wedge \inc_p {\mathbf v} ) \\
 & = \mu_{\onesp} \left( {\mathbf X}_L \wedge {\mathbf X}_R \wedge d \pi^r \left(  \inc_p^l {\mathbf X}_R^{S_R} \wedge  \inc_p {\mathcal C}_p^r {\mathbf X}_L^{S_L} \right) \right) \\
 & \qquad \qquad \cdot  
  \mu_{\onesp} ( {\mathbf X}_L \wedge {\mathbf X}_R \wedge d \pi^l  \inc_p {\mathbf v} ) \\
   & = \mu_{\onesp} \left( {\mathbf X}_L \wedge {\mathbf X}_R \wedge d \pi^r \inc_p \left(  {\mathcal C}_p^l {\mathbf X}_R^{S_R} \wedge  {\mathcal C}_p^r {\mathbf X}_L^{S_L} \right) \right)
  \mu_{\onesp} ( {\mathbf X}_L \wedge {\mathbf X}_R \wedge d \pi^l  \inc_p {\mathbf v} ).
 \end{align*}

Observe that $\mu_{\onesp}( {\mathbf X}_L \wedge {\mathbf X}_R \wedge {\mathbf v})$ and $\mu_{\onesp} ( {\mathbf X}_L \wedge {\mathbf X}_R \wedge d \pi^l  \inc_p {\mathbf v} )$ are both densities (as a function of $\mathbf v$) on $T_{\pi^c(p)}(\onesp) / (\ker d \pi_L + \ker d \pi_R)$. In particular, they differ at most by a multiplicative constant, and moreover this constant must equal $1$ when $\pi^c(p) = \pi^l(p)$. By similar reasoning, it follows that the quantity
\[ C_\inc := \frac{\mu_{\mathcal M} \left({\mathbf X}_R \wedge {\mathbf X}_L \wedge d \pi^l \inc_p {\mathbf v}  \right) \mu_{\mathcal M} \left({\mathbf X}_R \wedge {\mathbf X}_L \wedge d \pi^r \inc_p {\mathbf v} \right) }{\left( \mu_{\mathcal M}( {\mathbf X}_L \wedge {\mathbf X}_R \wedge {\mathbf v}) \right)^2 } \]
is independent of ${\mathbf v}$, equal to $1$ on $\pi^l(p) = \pi^c(p) = \pi^r(p)$ and
 \begin{align*}
 \mu &_{\Lsp \times \Rsp}  d \Pi  \left(  \inc_p^l {\mathbf X}_R^{S_R}  \wedge \inc_p^r (  {\mathbf X}_L^{S_L} \wedge {\mathbf X}_R ) \wedge ( \inc_p^l {\mathbf X}_L \wedge \inc_p {\mathbf v}) \right) \\
  & = C_\inc \mu_{\onesp} \left( {\mathbf X}_L \wedge {\mathbf X}_R \wedge   C_p^l {\mathbf X}_R^{S_R} \wedge  {\mathcal C}_p^r {\mathbf X}_L^{S_L} \right)
  \mu_{\onesp} ( {\mathbf X}_L \wedge {\mathbf X}_R \wedge  {\mathbf v} ).
\end{align*}
For simplicity, let us normalize each $v^i$ so that $\mu_{\onesp} ( {\mathbf X}_L \wedge {\mathbf X}_R \wedge  {\mathbf v} ) = 1$.

 Now by the smooth coarea formula \eqref{coarea0}, there is a density $\mu^\Pi$ on the nondegenerate fibers of $\Pi$ (meaning only those points at which $d \Pi$ is surjective) which 
 \[ \int_{\threesp} F d \mu = \int_{\Lsp \times \Rsp} \left[ \int_{\Pi^{-1}(x_L,x_R)} F d \mu^{\Pi} \right] d \mu_{\Lsp}(x_L) d \mu_{\Rsp}(x_R) \]
 when $F$ is any integrable function equaling zero on the set where $d \Pi$ is not surjective. This density $\mu_{\Pi}$ must satisfy
 \[ \mu ( { \mathbf \Xi} \wedge {\mathbf P}) = \mu_{\Lsp \times \Rsp} ( d \Pi ( {\mathbf \Xi})) \mu^{\Pi} ( {\mathbf P}) \]
for any unsigned decomposable $\kappa$-multivector field ${\mathbf P}$ generated by the vectors tangent to the fibers of $\Pi$ and any
unsigned decomposable $(n_L+n_R)$-multivector field ${\mathbf \Xi}$.
It follows that
\begin{align*}
 \mu^{\Pi}({\mathbf P})&  \mu_{\onesp} \left( {\mathbf X}_L \wedge {\mathbf X}_R \wedge   C_p^l {\mathbf X}_R^{S_R} \wedge  {\mathcal C}_p^r {\mathbf X}_L^{S_L} \right)
\\
& = C_\inc^{-1} \mu  \left(  \inc_p^l {\mathbf X}_R^{S_R}  \wedge \inc_p^r (  {\mathbf X}_L^{S_L} \wedge {\mathbf X}_R ) \wedge ( \inc_p^l {\mathbf X}_L \wedge \inc_p {\mathbf v}) \wedge {\mathbf P} \right) \\
& = C_\inc^{-1} \mu_{\onesp} d \pi^c \left(  \inc_p^l {\mathbf X}_R^{S_R}  \wedge \inc_p^r   {\mathbf X}_L^{S_L}   \wedge \inc_p {\mathbf v} \wedge {\mathbf P} \right).
\end{align*}
Notice that the top line does not depend on the choice of $\inc$, so the density on fibers  as written on the bottom line must also be independent of the choice of $\inc$. Moreover, it is always possible to choose $\inc$, defined in terms of vector fields $Z_i$ as done following \eqref{choices}, for which $C_\inc$ is nonzero at any particular point (simply require $d \pi^c {\mathbf Z} \wedge {\mathbf X}_L \wedge {\mathbf X}_R \neq 0$ and likewise $d \pi^c {\mathbf Z} \wedge d \pi_L^{-1}|_{\pi^c(p)} (d \pi_L({\mathbf X}_R|_{\pi^l(p)})) \wedge {\mathbf X}_L \neq 0$ and $d \pi^c ( \mathbf Z) \wedge  d \pi_R^{-1}|_{\pi^c(p)} (d \pi_R ({\mathbf X}_L|_{\pi^r(p)})) \wedge {\mathbf X}_R \neq 0$). Thus we may assume that $C_\inc$ is nonvanishing. In particular, if we define the weight function
\begin{equation} \label{keq}
{\mathcal K} (p) := \left[ \sum_{S_L,S_R} \left[  \mu_{\mathcal M} \left( {\mathbf X}_L \wedge {\mathbf X}_R \wedge {\mathcal C}_p^l {\mathbf X}_{R}^{S_R} \wedge {\mathcal C}_p^r {\mathbf X}_{L}^{S_L} \right) \right]^2 \right]^{\frac{1}{2}}, 
\end{equation}
and density
\begin{equation} \mu_0 ( {\mathbf P}) := C_\inc^{-1}  \left[ \sum_{S_L,S_R}  \left[ \mu_{\onesp} d \pi^c \left(  \inc_p^l {\mathbf X}_R^{S_R}  \wedge \inc_p^r   {\mathbf X}_L^{S_L}   \wedge \inc_p {\mathbf v} \wedge {\mathbf P} \right) \right]^2 \right]^{\frac{1}{2}}, \label{densityeq}
\end{equation}
then we have the identity
\begin{equation} \int F d \mu = \int_{\Lsp \times \Rsp} \left[ \int_{\Pi^{-1}(x_L,x_R)}  F \frac{d \mu_0}{{\mathcal K}} \right]  d \mu_\Lsp (x_L) d \mu_\Rsp(x_R) \label{coarea} \end{equation}
whenever $F$ is supported away from the set where ${\mathcal K} = 0$. We also note explicitly that $\mu_0$ is a measure of smooth density that only depends on the projection $\pi^c$-projection of the fibers of $\Pi$ rather than on the full fiber in $\threesp$. 

\subsection{Proof of \eqref{btrick} up to geometric sublevel set estimates}

In this section we complete the $L^1-L^1$-type estimates for the $TT^*T$ functional \eqref{quadlinear} using the coarea formula \eqref{coarea} and then explain how \eqref{btrick} follows if estimates of a certain geometric sublevel set functional are known to hold.
Recall that we may regard the manifold ${\mathcal M}$ to be a subset of $\Lsp \times \Rsp$ by identifying the point $m \in \onesp$ with the point  $(\pi_L(m),\pi_R(m)) \in \Lsp \times \Rsp$. Likewise, for any point $x = (x_L,x_R) \in \Lsp \times \Rsp$, the points in the set $\pi^c (\Pi^{-1}(x))$ are identified with the set
\[ \set{ (y_L,y_R) \in \Lsp \times \Rsp}{ \exists m \in  \Pi^{-1}(x_L,x_R) ~ y_L = \pi^c_L(m), ~ y_R = \pi^c_R(m)}. \]
In general the set $\pi^c \Pi^{-1}(x)$ will be an immersed $\kappa$-dimensional submanifold, although it may be possible that $d \Pi$ is not surjective at some points $m \in \Pi^{-1}(x)$. To avoid confusion, let $\Pi^{-1}_0(x)$ consist only of those points in $\Pi^{-1}(x)$ at which $d \Pi$ is surjective, and define $\pi^c \Pi_0^{-1}(x)$ analogously to $\pi^c \Pi^{-1}(x)$.
In Section \ref{solsec} we will establish that the $\mu_0$-measure of each ball $B_r(x) \subset \Lsp \times \Rsp$ intersected with the fiber $\pi^{c} \Pi^{-1}_0 (x)$ is controlled by some uniform constant times $r^{\kappa}$ when $x$ ranges over any compact set. (When convenient, we will abuse notation as we have just done and consider $\mu_0$ to be defined either on $\threesp$ or on $\pi^c(\threesp)$). Consequently, when $F$ and $G$ are Borel measurable sets in $\Lsp$ and $\Rsp$ with bounded diameter, it follows that
\begin{align}
 \int_{\pi^{c} \Pi^{-1}_0 (x) \cap (F \times G)} &  \frac{d \mu_0(m)}{(\dist (m, x))^\kappa}  \leq \sum_{j \in \Z} 2^{-\kappa j} \int_{\pi^{c} \Pi^{-1}_0 (x) \cap (F \times G) \cap B_{2^j} (x)} d \mu_0  \nonumber \\ & \lesssim  \ln \left( 2 +\frac{\diam (F \times G)}{\dist(x, \pi^c \Pi^{-1}_0(x) \cap (F \times G))} \right) \label{intest}
 \end{align}
where $\dist(m,x)$ is distance in $\Lsp \times \Rsp$ as measured by the standard metric on that space.

Recall from \eqref{quadlinear} that
\begin{equation} B^{(3)}_{G'F'}(\chi_{G},\chi_{F}) = 
\int_{\threesp}  ( \chi_{F} \circ \pi^r_L) (\chi_{G'} \circ \pi^c_R) (\chi_{F'} \circ \pi^c_L) (\chi_{G} \circ \pi^l_R) d \mu. \label{ttst} \end{equation}
Fix any positive real number $\alpha$; we will estimate the right-hand side of \eqref{ttst} when the domain of integration is restricted to the set $$S_\alpha := \set{ p \in \threesp}{{\mathcal K}(p) (\dist(\Pi(p),\pi^c(p))) ^{-\kappa} > \alpha}.$$ It follows from \eqref{coarea} that
\begin{align}
 \int_{S_\alpha} &  ( \chi_{F} \circ \pi^r_L) (\chi_{G'} \circ \pi^c_R) (\chi_{F'} \circ \pi^c_L) (\chi_{G} \circ \pi^l_R) d \mu  \nonumber \\ 
  = & \int_{F \times G} \left[ \int_{\Pi^{-1}(x_L,x_R)}  (\chi_{G'} \circ \pi^c_R) (\chi_{F'} \circ \pi^c_L) \chi_{S_\alpha}  \frac{d \mu_{0}}{\mathcal K} \right] d \mu_{\Lsp}(x_L) d \mu_\Rsp(x_R) \label{ncalc1} 
  \end{align}
  which applies because $S_\alpha$ is disjoint from the set where $\mathcal K$ is zero.
  Ideally, we would like to show that the integrand in brackets on the right-hand side \eqref{ncalc1} is uniformly bounded above by $C \alpha^{-1}$. This turns out not to be the case, but it only fails logarithmically. To see this, we begin by further simplifying \eqref{ncalc1} by exploiting the definition of $S_\alpha$ to conclude that
  \begin{align*}
  \int_{S_\alpha} &  ( \chi_{F} \circ \pi^r_L) (\chi_{G'} \circ \pi^c_R) (\chi_{F'} \circ \pi^c_L) (\chi_{G} \circ \pi^l_R) d \mu \\ 
  \leq & \frac{1}{\alpha} \int_{F \times G} \left[ \int_{\pi^c \Pi^{-1}(x_L,x_R) \cap (F' \times G')}   \frac{d \mu_{0}(m)}{(\dist(m,(x_L,x_R))^\kappa} \right] d \mu_{\Lsp}(x_L) d \mu_\Rsp(x_R).
  \end{align*}
  We estimate the integrand using \eqref{intest} to conclude that
  \begin{align*}
  \int_{S_\alpha} &  ( \chi_{F} \circ \pi^r_L) (\chi_{G'} \circ \pi^c_R) (\chi_{F'} \circ \pi^c_L) (\chi_{G} \circ \pi^l_R) d \mu  \\
  \lesssim & \frac{1}{\alpha} \int_{F \times G} \left[  \ln \left( 2 +\frac{\diam (F \times G)}{\dist(x, \pi^c \Pi^{-1}_0(x) \cap (F \times G))} \right) \right] d \mu_{\Lsp} d \mu_\Rsp \\
   \lesssim & \frac{1}{\alpha} \int_{F \times G} \left[  \ln \left( 2 +\frac{\diam (F \times G)}{\dist(x, \onesp \cap (F \times G))} \right) \right] d \mu_{\Lsp} d \mu_\Rsp
 \end{align*}
 (where we can replace $F'$ and $G'$ on the right-hand side with $F$ and $G$ since $F' \subset F$ and $G' \subset G$).
 Assuming that $F$ and $G$ are supported in a sufficiently small ball, the set of points $x$ in $\Lsp \times \Rsp$ at distance $\delta$ to $\onesp \cap (F \times G)$ will have measure controlled by $\delta^{\ell} (\diam(F \times G))^{n_L + n_R - \ell}$. Consequently, fixing any $\delta > 0$ and breaking the integral into pieces on which $\dist(x, \onesp \cap (F \times G)) \geq \delta$ and $\dist (x, \onesp \cap (F \times G)) \in [2^{-j-1} \delta, 2^{-j} \delta]$ for $j = 0,1,2,\ldots$, we will have that
 \begin{align*}
  \int_{F \times G} & \left[  \ln \left( 2 +\frac{\diam (F \times G)}{\dist(x, \onesp \cap (F \times G))} \right) \right] d \mu_{\Lsp} d \mu_\Rsp \\ & \leq \ln \left( 2 + \frac{\diam (F \times G)}{\delta} \right) \left(  |F| |G| + \delta^{\ell} (\diam(F \times G))^{n_L + n_R - \ell} \right), 
  \end{align*}
  and choosing $\delta$ appropriately gives
  \begin{align*}
  \int_{S_\alpha} &  ( \chi_{F} \circ \pi^r_L) (\chi_{G'} \circ \pi^c_R) (\chi_{F'} \circ \pi^c_L) (\chi_{G} \circ \pi^l_R) d \mu \\ 
  & \lesssim \frac{1}{\alpha} |F| |G| \ln \left( 2 + \frac{ (\diam (F \times G))^{\frac{n_L + n_R}{\ell}}}{ (|F| |G|)^{\frac{1}{\ell}}} \right).
  \end{align*}
In particular, for  any $\epsilon > 0$, it must be the case that
\begin{equation}
\begin{split}
 \int_{S_\alpha}   ( \chi_{F} \circ \pi^r_L) (\chi_{G'} \circ \pi^c_R) & (\chi_{F'} \circ \pi^c_L) (\chi_{G} \circ \pi^l_R) d \mu \\ & \lesssim \frac{1}{\alpha} (\diam(F \times G))^{(n_L + n_R) \epsilon}  |F|^{1-\epsilon} |G|^{1-\epsilon} 
 \end{split}
 \label{part1}
 \end{equation}
uniformly in $\alpha$ and the sets $F,G,F'$ and $G'$. This is exactly the first term on the right-hand side of \eqref{btrick}, since the diameters of $F \times G$ is assumed to be bounded. More precisely, we have shown that
\begin{align*}
 B^{(3)}_{G'F'}(\chi_G,\chi_F) & \lesssim \alpha^{-1} |F|^{1-\epsilon} |G|^{1-\epsilon} + \\
&  \int_{\threesp \setminus S_\alpha} \! \! \left( \chi_G \circ \pi_R^l \right) \left( \chi_{F'} \circ \pi_L^c \right) \left(  \chi_{G'} \circ \pi_R^c \right) \left(  \chi_F \circ \pi_L^r \right)  d \mu,
\end{align*}
and what remains is to understand the integral over $\threesp \setminus S_\alpha$ as a sublevel set functional and use the bounds on that sublevel set functional to complete an inequality of the form \eqref{btrick}.
%
%\subsection{Reduction of singular Jacobian piece to a sublevel set estimate}
%
%As just mentioned, to complete the proof of \eqref{btrick} in light of \eqref{part1}, it suffices to estimate the integral
%\begin{align*}
% \int_{\frac{{\mathcal K(p)}}{(\dist(\pi^c(p),\Pi(p)))^{\kappa}} \leq \alpha}  ( \chi_{F} \circ \pi^r_L) (\chi_{G'} \circ \pi^c_R) (\chi_{F'} \circ \pi^c_L) (\chi_{G} \circ \pi^l_R) d \mu.
% \end{align*}
 For convenience, let us define
 \begin{equation} \Phi(p) := \frac{{\mathcal K(p)}}{(\dist(\pi^c(p),\Pi(p)))^{\kappa}}. \label{phieq} \end{equation}
 By Fubini, we may write
 \begin{align*}
 \int_{\Phi(p) \leq \alpha} &  ( \chi_{F} \circ \pi^r_L) (\chi_{G'} \circ \pi^c_R) (\chi_{F'} \circ \pi^c_L) (\chi_{G} \circ \pi^l_R) d \mu(p) \\ 
 & =  \int_{\onesp \cap \pi_L^{-1} (F') \cap \pi_R^{-1} (G') }  W_{\alpha,m^c} (\chi_{G} \circ \pi_R, \chi_{F} \circ \pi_L )  d \mu_{\mathcal M} (m^c)
 \end{align*}
 where $W_{\alpha,m^c}(g,f)$ is the bilinear sublevel set functional
\begin{equation} 
\begin{split} W_{\alpha,m^c}& (g,f) := \\ & \int_{\pi_R^{-1} \pi_R(m^c)}  \int_{\pi_L^{-1} \pi_L(m^c)}  \chi_{\Phi(p) \leq \alpha} g (m^l) f (m^r)  d \mu^{\pi_L} (m^l)  d \mu^{\pi_R} (m^r),   
\end{split} \label{sublevels} \end{equation}
where $p = (m^l,m^c,m^r)$.
In particular, if $W_{\alpha,m^c}$ satisfies a restricted weak-type estimate
\[ W_{\alpha,m^c} (\chi_{E^l}, \chi_{E^r} )  \lesssim \alpha^s |\mu^{\pi_L}(E^l)|^{\frac{1}{p_l}} | \mu^{\pi_R} (E^r)|^{\frac{1}{p_r}} \]
 uniformly in $\alpha$ and $m^c$ for all measurable sets $E^l$ and $E^r$ contained in $\pi_L^{-1} \pi_L(m^c)$ and $\pi_R^{-1} \pi_R(m^c)$ (i.e., belonging to the fibers of $\pi_L$ and $\pi_R$ passing through $m^c$), respectively, then it follows that
 \begin{align*}
 \int_{\Phi(p) \leq \alpha}&   ( \chi_{F} \circ \pi^r_L) (\chi_{G'} \circ \pi^c_R) (\chi_{F'} \circ \pi^c_L) (\chi_{G} \circ \pi^l_R) d \mu(p) \\
 & \lesssim \alpha^s \int |T_{GF}^* \chi_G |^{\frac{1}{p_l}} T^*_{G'F'} |T \chi_{GF} \chi_F|^{\frac{1}{p_r}},
 \end{align*}
which combines with \eqref{part1} to prove \eqref{btrick}. Thus the question of boundedness of the functional \eqref{bilinear} is reduced to the study of \eqref{sublevels}. Since we have made no use up to this point of any notion of nondegeneracy whatsoever, the geometry of \eqref{bilinear} is now entirely captured by the sublevel set operator \eqref{sublevels}. In particular, this means that if the bilinear form \eqref{bilinear} fails to exhibit any meaningful curvature, there is no reason to expect that \eqref{sublevels} will satisfy any nontrivial estimates.  The problem of proving boundedness of \eqref{sublevels} using the geometry of \eqref{bilinear} is taken up in the following section after an auxiliary lemma is established to help prove sublevel set functional inequalities in a systematic way.

\section{Applications and examples}
\label{calculationsec}
\subsection{An auxiliary lemma for establishing sublevel set estimates}

A particularly successful general strategy for proving restricted weak type estimates for sublevel set functionals like \eqref{sublevels} is to independently decompose both $m^l$ and $m^r$ into dyadic annuli with centers at the point $m^c$, since, in particular, we know that ${\mathcal K}(p)$ is expected to be identically zero when $m^l = m^c$ and when $m^r = m^c$. This is due to the fact that ${\mathcal C}_p^l$ and ${\mathcal C}_p^r$, respectively, reduce to the identity (modulo the sum of kernels), meaning that ${\mathcal C}_p^l {X}_R^j = 0$ when $m^l = m^c$ and ${\mathcal C}_p^r {X}_L^i = 0$ when $m^r = m^c$. The following lemma establishes a type of interpolation result which is particularly useful in this case. Roughly speaking, if there are two natural restricted weak-type inequalities which hold on each individual product of annuli (i.e., $||m^l - m^c|| \sim 2^i$ and $||m^r - m^c|| \sim 2^j$), then under appropriate technical hypotheses, the interpolated restricted weak-type inequalities hold not just for individual annuli, but for the sum over all annuli as well. Readers will note that the argument is essentially the same one that appears in the proof of the Marcinkiewicz interpolation theorem.
\label{sublevellemmasec} 
\begin{lemma}
Suppose $a_0,a_1,b_0,b_1$ are real numbers and $p_0,p_1,q_0,q_1 \in [1,\infty]$ are exponents satisfying \label{sublevellemma} 
\begin{equation} \det \left[ \begin{array}{cc} a_0 & p_0^{-1}  \\ a_1 & p_1^{-1} \end{array} \right] \neq 0,   \ \  \det \left[ \begin{array}{cc} b_0 & q_0^{-1}  \\ b_1 & q_1^{-1} \end{array} \right] \neq 0, \label{detcon} \end{equation}
\begin{equation} \frac{1}{p_k} + \frac{1}{q_k} \geq 1, ~ k = 0,1. \label{mconst} \end{equation}
Then for any $\theta \in (0,1)$, let $a_\theta,b_\theta, p_\theta,q_\theta$ be defined by the formulas 
\[ \frac{1}{p_\theta} := \frac{1-\theta}{p_0} + \frac{\theta}{p_1} \mbox{ and } \frac{1}{q_\theta} := \frac{1-\theta}{q_0} + \frac{\theta}{q_1}, \]
\[ a_\theta := (1-\theta) a_0 + \theta a_1 \mbox{ and }  b_\theta := (1-\theta) b_0 + \theta b_1. \]
There exists a constant $C$ depending only on the choices of each $a_k,b_k,p_k,q_k$, and $\theta$ such that
\begin{equation}
\begin{split}
 \sum_{i,j \in \Z}  \min & \left\{  A_0  2^{a_0 i + b_0 j} |f_i|^{\frac{1}{p_0}} |g_j|^{\frac{1}{q_0}}, A_1 2^{a_1 i + b_1 j} |f_i|^{\frac{1}{p_1}} |g_j|^{\frac{1}{q_1}} \right\}  \label{marc} \\
 &  \leq C A_0^{1-\theta} A_1^{\theta} \left( \sum_{i \in \Z} 2^{a_\theta p_\theta i} |f_i| \right)^{\frac{1}{p_\theta}} \left( \sum_{j \in \Z} 2^{b_\theta q_\theta j} |g_j| \right) ^{\frac{1}{q_\theta}} 
 \end{split}
 \end{equation}
for any nonnegative constants $A_0$ and $A_1$ and any sequences $\{f_i\}_{i \in \Z}, \{g_j\}_{j \in \Z}$ of real or complex numbers.

\end{lemma}
\begin{proof}
By the definition of $a_\theta$, $b_\theta$, $p_\theta$, and $q_\theta$, we have the identities
\begin{align*}
(1-\theta) \left( a_0 - \frac{a_\theta p_\theta}{p_0} \right) & + \theta \left( a_1 - \frac{a_\theta p_\theta}{p_1} \right) = 0, \\
(1-\theta) \left( b_0 - \frac{b_\theta q_\theta}{q_0} \right) & + \theta \left( b_1 - \frac{b_\theta q_\theta}{q_1} \right) = 0;
\end{align*}
by \eqref{detcon}, it is not possible to choose $p_0^{-1} = p_1^{-1} = 0$, so $p_\theta$ must be finite. Likewise $q_\theta < \infty$. Moreover it cannot be the case that both $a_0 - a_\theta p_\theta p_0^{-1}$ and $a_1 - a_\theta p_\theta p_1^{-1}$ are zero, since this would also force the first determinant \eqref{detcon} to be zero. Consequently neither is zero (since their convex combination vanishes). Likewise neither $b_0 - b_\theta q_\theta q_0^{-1}$ nor $b_1 - b_\theta q_\theta q_1^{-1}$ is zero, and since $\theta \in (0,1)$, one must be positive and the other negative. Rewriting the left-hand side of \eqref{marc}  in terms of sequences
\[ \tilde f_i := 2^{a_\theta p_\theta i} f_i \mbox{ and } \tilde g_j := 2^{b_\theta q_\theta j} g_j, \]
it suffices to assume that $a_\theta = b_\theta = 0$ and that none of $a_0,b_0,a_1,b_1$ is zero. This amounts to replacing $a_k$ by $a_k - \frac{a_\theta p_\theta}{p_k}$ and likewise for $b_k$. These changes preserve the value of the determinants \eqref{detcon}.
 By symmetry, it also suffices to assume that $b_0 > 0 > b_1$.

Let $\tau = - a_0 b_0^{-1}$, and let $s$ be any real number (to be fixed shortly). Now
\begin{align*}
 \sum_{j \in \Z} &  \min \left\{  A_0  2^{a_0 i + b_0 j} |f_i|^{\frac{1}{p_0}} |g_j|^{\frac{1}{q_0}}, A_1 2^{a_1 i + b_1 j} |f_i|^{\frac{1}{p_1}} |g_j|^{\frac{1}{q_1}} \right\} \\
= & A_0  2^{b_0 s} \! \! \sum_{j < \tau i+ s} A_0  2^{a_0 i + b_0 j- b_0 s}  |g_j|^{\frac{1}{q_0}} 
+ A_1 2^{b_1 s} |f_i|^{\frac{1}{p_1}} \! \! \sum_{j \geq \tau i + s} 2^{a_1 i + b_1 j - b_1 s}  |g_j|^{\frac{1}{q_1}}.
 \end{align*}
Consider the mappings $T_0$ and $T_1$ which act on sequences as follows:
\[ (T_0 e)_i := \sum_{j < \tau i +  s } 2^{a_0 i + b_0 j - b_0 s} e_j \mbox{ and } (T_1 e)_i := \sum_{j \geq \tau i +  s } 2^{a_1 i + b_1 j - b_1 s} e_j. \]
Using standard sum estimation techniques (together with the fact that $b_0 > 0 > b_1$), we have that
\begin{align*}
 || T_0 e||_\infty \leq&  \frac{||e||_\infty}{1-2^{-b_0}} , \ ||T_0 e||_1 \leq \frac{||e||_1}{1-2^{-|a_0|}} , \\ 
   || T_1 e||_\infty \leq & \frac{||e||_\infty}{1-2^{b_1}}  , \ ||T_1 e||_1 \leq \frac{||e||_1}{1-2^{-|a_1|}},
 \end{align*}
and, in particular, $T_i$ is bounded on $\ell^{p_i'}$ with constant independent of $s$, where $p_i'$ is dual to $p_i$. Therefore
\begin{align*}
\sum_{i \in Z} &  |f_i|^{\frac{1}{p_0}} \sum_{j < \tau i+s} 2^{b_0 (j-\tau i-s)} |g_j|^{\frac{1}{q_0}} \lesssim ||f||_1^{\frac{1}{p_0}} || |g|^{\frac{1}{q_0}} ||_{p_0'}^{\frac{1}{p_0'}} \leq ||f||_1^{\frac{1}{p_0}} || g ||_1^{\frac{1}{q_0}}, \\
\sum_{i \in \Z} & |f_i|^{\frac{1}{p_1}} \sum_{j \geq \tau i+s} 2^{b_1 (j-\tau i-s)} |g_j|^{\frac{1}{q_1}} \lesssim ||f||_1^{\frac{1}{p_1}} ||g||^{\frac{1}{q_1}}_1
\end{align*}
since $p_k' \geq q_k$ for $k = 0,1$. Therefore
\begin{align*}  \sum_{i,j \in \Z} &   \min \left\{  A_0  2^{a_0 i + b_0 j} |f_i|^{\frac{1}{p_0}} |g_j|^{\frac{1}{q_0}}, A_1 2^{a_1 i + b_1 j} |f_i|^{\frac{1}{p_1}} |g_j|^{\frac{1}{q_1}} \right\}\\ &  \lesssim A_0 2^{b_0 s} ||f||_1^{\frac{1}{p_0}} || g ||_1^{\frac{1}{q_0}} + A_1 2^{b_1 s} ||f||_1^{\frac{1}{p_1}} ||g||^{\frac{1}{q_1}}_1. 
\end{align*}
Optimizing over the choice of $s$ gives
\begin{align*}  \sum_{i,j \in \Z} &   \min \left\{  A_0  2^{a_0 i + b_0 j} |f_i|^{\frac{1}{p_0}} |g_j|^{\frac{1}{q_0}}, A_1 2^{a_1 i + b_1 j} |f_i|^{\frac{1}{p_1}} |g_j|^{\frac{1}{q_1}} \right\}\\ &  \lesssim \left(  A_0  ||f||_1^{\frac{1}{p_0}} || g ||_1^{\frac{1}{q_0}} \right)^{1-\theta}   \left( A_1 ||f||_1^{\frac{1}{p_1}} ||g||^{\frac{1}{q_1}}_1 \right)^{\theta}
\end{align*}
with $\theta = b_0/(b_0 - b_1)$, which happens to be the correct value of $\theta$  to give $b_\theta = 0$.
\end{proof}

A final remark about the lemma: although it will not be needed here, the constraint \eqref{mconst} can be weakened somewhat. In particular, if $p_\theta^{-1} + q_\theta^{-1} > 1$ for the chosen value of $\theta$, then one does not need to explicitly assume \eqref{mconst}, since one can instead apply the lemma using some convex combination of the estimates on the left-hand side of \eqref{marc}. It will always be possible in this setting to find two different convex combinations which automatically both satisfy \eqref{mconst} and still yield the same conclusion \eqref{marc} for the desired exponents $p_\theta$ and $q_\theta$.

\subsection{General bilinear averages}

We now begin the study of the geometric sublevel set operators \eqref{sublevels} in earnest. The first case to be considered corresponds to the setting of Theorem \ref{bilinearthm} for the averaging operators constructed from bilinear mappings.
Let  $Q \ : \  \R^{d_L} \times \R^{d_R} \rightarrow \R^{\ell}$ be any such bilinear map.  Let $\Omega \subset \R^{d_L + d_R + \ell}$ be compact, fix $\onesp$ to be any open set containing $\Omega$, and consider the bilinear functional 
\[ B_Q(f,g) := \int_{\Omega} f(y,z+Q(x,y)) g(x,z) dx dy dz. \]
By duality, this bilinear functional corresponds to the integral operator
\[ Tf(x,z) := \int_{\Omega_{x,z}} f(y,z+Q(x,y)) dy. \]
If we define $\pi_L (x,y,z) := (y,z+Q(x,y))$ and $\pi_R(x,y,z) := (y,z)$, one can easily identify vector fields $X_L^i$ and $X_R^j$ annihilated by $d \pi_L$ and $d \pi_R$, respectively:
\[ X_L^i := \frac{\partial}{\partial x_i} - Q(e_i,y) \cdot \nabla_z \mbox{ and } X_R^j := \frac{\partial}{\partial y_j}, \]
where the vectors $e_i$ and $e_j$ denote standard basis vectors in $\R^{d_L}$ and $\R^{d_R}$, respectively. On the space $\threesp$ we can calculate that
\begin{align*}
 d \pi_L^l \left( \frac{\partial}{\partial y_j^\ell} \right)  = & d \pi_L^c \left( \frac{\partial}{\partial y_j^c} + Q (x^l - x^c,e_j) \cdot \nabla_{z^c} \right), \\
 d \pi_R^r \left( \frac{\partial}{\partial x_i^r} - Q(e_i,y^r) \cdot \nabla_{z^r} \right) = & d \pi_R^c \left( \frac{\partial}{\partial x_i^c} - Q(e_i,y^c) \cdot \nabla_{z^c} \right. \\ & \qquad \qquad \left. \vphantom{\frac{\partial}{\partial x_i^c}} + Q(e_i,y^c-y^r) \cdot \nabla_{z^c} \right).
 \end{align*}
Since ${\mathcal C}_p^l$ and ${\mathcal C}_p^r$ from \eqref{lproj} and \eqref{rproj} are intrinsically defined modulo $\ker d \pi_L + \ker d \pi_R$ (not depending on $\inc_p$), it follows that
\[ {\mathcal C}_p^l (X_R^j) \sim  Q( x^l - x^c, e_j) \cdot \nabla_{z^c} \mbox{ and }  {\mathcal C}_p^r (X_L^i) \sim Q(e_i,y^c - y^r)  \cdot \nabla_{z^c}. \]
If we take the standard measure $\mu_{\onesp} = d x dy dz$, then by \eqref{keq} we have that $({\mathcal K}(p))^2$ is the sum of squares of all $\ell \times \ell$ determinants whose columns are of the form $Q(x^l - x^c, e_j)$ or $Q(e_i,y^c-y_r)$ as $e_j$ and $e_i$ range over all elements of the standard bases. Consequently if we define
\[ \Phi_Q(x,y) = \frac{\mathrm{Vol}( \{ Q( e_i, y) \}_{i=1}^{d_L}, \{ Q(x,e_j) \}_{j=1}^{d_R} )}{ ( |x|^2 + |y|^2)^{\frac{d_L + d_R - \ell}{2}}},  \]
then the sublevel set operator \eqref{sublevels} is nearly translation invariant under the map $(m^l,m^c,m^r) \mapsto (m^l + \tau, m^c + \tau, m^r + \tau)$, with the only failure of invariance coming implicitly through the fact that ${\onesp}$ may not contain all of $\R^{d_L + d_R + \ell}$. Consequently, to prove estimates for \eqref{sublevels} uniformly in $m^c$, it suffices to bound the fixed sublevel set operator
\begin{equation} W_{\alpha}(g,f) = \int_{\R^{d_L} \times \R^{d_R}}  \chi_{\Phi_Q(x,y) \leq \alpha} g(x^l) f(y^r) dx dy. \label{bilinearsub} \end{equation}
Moreover, even if $g$ and $f$ in \eqref{bilinearsub} are restricted to fixed neighborhoods of the origin, it will still be the case that $W_{\alpha,m^c}$ is dominated by $W_{\alpha}$ provided that the diameter of $\Omega$ is sufficiently small. Thus by \eqref{part1}, \eqref{sublevels}, and Lemma \ref{mainlemma}, Theorem \ref{bilinearthm} must hold.

\subsection{The maximal 2D quadratic surface in $\R^5$ and generalizations}
\label{reductionsec}

In the case of the bilinear functional
\begin{equation} B(f,g) :=  \int_{\Omega} f \left( x_1 + t_1, x_2 + t_2 , x_{3} + \frac{1}{2}  t_1^2 , x_4 + t_1 t_2 , x_5 + \frac{1}{2} t_2^2 \right) g(x) dx dt \label{b2d} \end{equation}
corresponding to the operator \eqref{modelmaxml}, we choose vector fields $X^i_L$ and $X^j_R$ as follows:
\[ X_L^1 := \frac{\partial}{\partial t_1} - \frac{\partial}{\partial x_1} - t_1 \frac{\partial}{\partial x_3} - t_2 \frac{\partial}{\partial x_4},   X_L^2 := \frac{\partial}{\partial t_2} - \frac{\partial}{\partial x_2} - t_1 \frac{\partial}{\partial x_4} - t_2 \frac{\partial}{\partial x_5}, \] 
\[ X_R^j := \frac{\partial}{\partial t_j}, j=1,2. \]
Using these vector fields, the maps ${\mathcal C}_p^l$ and ${\mathcal C}_p^r$, defined by \eqref{lproj} and \eqref{rproj}, give
\begin{align*}
 {\mathcal C}^l_p X^1_R & \sim (t_1^l - t_1^c) \frac{\partial}{\partial x_3^c} + (t_2^l - t_2^c) \frac{\partial}{\partial x_4^c}, ~ ~ {\mathcal C}^l_p X^2_R  \sim (t_1^l - t_1^c) \frac{\partial}{\partial x_4^c} + (t_2^l - t_2^c) \frac{\partial}{\partial x_5^c}, \\
  {\mathcal C}^r_p X^1_L & \sim (t_1^c - t_1^r) \frac{\partial}{\partial x_3^c} + (t_2^c - t_2^r) \frac{\partial}{\partial x_4^c}, ~ ~ {\mathcal C}^r_p X^2_L  \sim (t_1^c - t_1^r) \frac{\partial}{\partial x_4^c} + (t_2^c - t_2^r) \frac{\partial}{\partial x_5^c}
  \end{align*}
 (with $\sim$ meaning modulo $(\ker d \pi_L + \ker d \pi_R)|_{(x^c,t^c)}$). Consequently by \eqref{keq},
\begin{equation} {\mathcal K}(p) = \left( ||t^l - t^c||^2 + ||t^r - t^c||^2 \right)^\frac{1}{2} \left| \det \left[ \begin{array}{cc} t_1^l - t_1^c & t_1^r - t_1^c \\ t_2^l - t_2^c & t_2^r - t_2^c \end{array} \right] \right|. \label{keq2d} \end{equation}
The fibers of the map $\Pi$ defined by \eqref{bigpi} are generically one-dimensional; consequently $\kappa = 1$ and by \eqref{phieq} and \eqref{sublevels}, it suffices to study 
\begin{equation} W_{\alpha,m^c} (g,f) := \int_{\Omega_{m^c}} \chi_{|\det (t^l-t^c \ t^r-t^c)| \leq \alpha} g(t^l) f(t^r) dt^l dt^r. \label{sublev00} \end{equation}
In particular, we will use \eqref{marc} to show that
\begin{equation} W_{\alpha,m^c}(\chi_{E^l},\chi_{E^r}) \leq C \alpha |E^l|^{\frac{1}{2}} |E^r|^{\frac{1}{2}}. \label{sublev01} \end{equation}
Without loss of generality, we may assume $t^c = 0$.  For each $j \in \Z$, let $A_j \subset \R^2$ be the annulus $\set{ t \in \R^2}{ 2^{j-1} \leq ||t|| < 2^j}$. To prove \eqref{sublev01}, we break \eqref{sublev00} into a sum over dyadic annuli in $t^l$ and in $t^r$, at which point it suffices by \eqref{marc} with $\theta = \frac{1}{2}$ to prove that
\begin{equation}
\begin{split}
 \int_{ |\det(t^l \ t^r)| \leq \alpha} &  \chi_{E^l \cap A_i} (t^l) \chi_{E^r \cap A_j}(t^r) dt^l dt^r \\ & \leq C \alpha \min\{ 2^{i-j} |E^r \cap A_j|, 2^{-i+j} |E^l \cap A_i| \}.
\end{split} \label{maxq1}
\end{equation}
Both estimates on the right-hand side follow from Fubini's theorem (in the former case integrating over $t^l$ first and in the latter integrating over $t^r$ first) and the inequalities
\begin{align}
 \left| \set{t^l \in \R^2}{ |\det (t^l \ t^r )| \leq \alpha \mbox{ and } 2^{i-1} \leq ||t^l|| < 2^i} \right| & \leq C \alpha 2^{i} ||t^r||^{-1}, \label{msublev1} \\
 \left| \set{t^r \in \R^2}{ |\det (t^l \ t^r )| \leq \alpha \mbox{ and } 2^{j-1} \leq ||t^r|| < 2^j} \right| & \leq C \alpha 2^{j} ||t^l||^{-1}, \label{msublev2}
 \end{align}
 which hold trivially since, for fixed $y \in \R^2$, the set $\set{ x \in \R^2}{ |\det (x \ y)| \leq \alpha}$
 is simply those points at distance less than $\alpha ||y||^{-1}$ to the line through the origin with direction $y$. Thus \eqref{sublev01} follows from \eqref{maxq1} and Lemma \ref{sublevellemma}.

In light of \eqref{sublev01}, by Lemma \ref{mainlemma} together with the trivial estimates
\[ |B (f,g)| \leq C ||f||_1 ||g||_\infty \mbox{ and } | B(f,g)| \leq C ||f||_\infty ||g||_1 \]
(assuming the compactness of $\Omega$), we have that
\[ |B(f,g)| \leq C ||f||_{q_L} ||g||_{q_R} \]
whenever $(\frac{1}{q_L},\frac{1}{q_R})$ belongs to the interior of the triangle with vertices $(1,0)$, $(0,1)$, and $(\frac{5}{8}, \frac{5}{8})$.

Incidentally, the compactness of $\Omega$ can be easily removed using translation invariance: For any $j \in \Z^{5}$, if we let $\Omega_j := [j_1,j_1+1] \times \cdots \times[j_5,j_5+1] \times [-1,1]^2$, we have uniformly in $j$ that
\[ |B_{\Omega_j}(f,g)| \leq C ||f||_{q_L} ||g||_{q_R}, \]
which self-improves to the estimate
\[ |B_{\Omega_j}(f,g)| \leq C ||f \chi_{Q_j^*} ||_{q_L} ||g \chi_{Q_j}||_{q_R} \]
where $Q_j := [j_1,j_1+1] \times \cdots \times[j_5,j_5+1]$ and $Q_j^* := [ j_1 - 1, j_1 + 2] \times \cdots \times [j_5 - 1, j_5 + 2]$.  Therefore, recalling $\gamma_M$ from the notation of \eqref{modelmaxml}, it must be the case that
\begin{align*}
 \left| \int_{\R^5} g(x) \right. & \left. \int_{[-1,1]^2} f(x + \gamma_M(t)) g(x) dt dx \right| \leq \sum_{j \in \Z^5} |B_{\Omega_j} (f,g)| \\
 & \leq \sum_{j \in \Z^5} C ||f \chi_{Q_j^*} ||_{q_L} ||g \chi_{Q_j}||_{q_R} \\
 & \leq C \left( \sum_{j \in \Z^5} ||f \chi_{Q_j^*}||_{q_L}^{q_L} \right)^\frac{1}{q_L} \left( \sum_{j \in \Z^5} ||g \chi_{Q_j}||_{q_R}^{q_R} \right)^{\frac{1}{q_R}} \\
 & \lesssim ||f||_{q_L} ||g||_{q_R}
 \end{align*}
by H\"{o}lder's inequality together with the observation that $q_L^{-1} + q_R^{-1} \geq 1$. By the Marcinkiewicz interpolation theorem, this completes the claims of Theorem \ref{modelthm} for the maximal quadratic submanifold in $\R^5$ given by \eqref{modelmaxml}.

Let us now turn to the geometric setting of Theorem \ref{variable5}.
When the dimension of $\onesp$ is seven and $d_L = d_R = 5$, will say that the pair of projections $\pi_L$ and $\pi_R$ are {\bf nondegenerate} when at every point $m \in \onesp$, the following conditions hold:
\begin{itemize}
\item For any $v,v' \in \R^2$ with $\det (v \ v') \neq 0$, there is an $a_{v,v'} \in \R^2$ such that
\begin{equation} {\mathbf X}_L \wedge {\mathbf X}_R \wedge [ X_L^1, v \cdot X_R] \wedge [X_L^2, v \cdot X_R] \wedge [a_{v,v'} \cdot X_L, v' \cdot X_R] \neq 0. \label{leftid} 
\end{equation}
\item For any $u,u' \in \R^2$ with $\det (u \ u') \neq 0$, there is a $b_{u,u'} \in \R^2$ such that
\begin{equation} {\mathbf X}_L \wedge {\mathbf X}_R \wedge [u \cdot X_l, X_R^1] \wedge [u \cdot X_L, X_R^{2}] \wedge [u' \cdot X_L, b_{u,u'} \cdot X_R] \neq 0. \label{rightid}
\end{equation}
\end{itemize}
(Note that the condition is independent of the choice of bases $\{X^1_L,X_L^2\}$ and $\{X^1_R,X^2_R\}$ of $\ker d \pi_L$ and $\ker d \pi_R$.) Theorem \ref{variable5} asserts that, under the regularity assumptions of real analyticity, any pair of nondegenerate projections as defined above gives rise to a bilinear functional \eqref{bilinear} which is bounded for pairs of exponents $(q_L^{-1},q_R^{-1})$ in the interior of the triangle with vertices $(1,0), (0,1)$, and $(\frac{5}{8},\frac{5}{8})$.

The proof of Theorem \ref{variable5} proceeds as follows.
For each fixed $m^c \in \onesp$, we can coordinatize a small neighborhood of $(m^c,m^c,m^c) \in \threesp$ by exponentiating as follows: for $u \in \R^{d_L}$ and $v \in \R^{d_R}$, we define
\[ p(m^c,u,v) := (\exp(u \cdot X_L)(m^c), m^c, \exp(v \cdot X_R)(m^c)). \]
The advantage of this representation is that we get an explicit approximation of ${\mathcal C}_p^l (X_R^j)$ and ${\mathcal C}_p^r(X_L^i)$ in terms of commutators. In particular, since ${\mathcal C}_p^l (X_R^j) = {\mathcal C}_p^r(X_L^i) = 0$ when $u = v = 0$, by \eqref{leftcurv}, \eqref{rightcurv}, and Taylor's theorem, we must have that
\begin{align*}
 {\mathcal C}_{p(m^c,u,v)}^l (X_R^j) &  = \hphantom{-} [ u \cdot X_L, X_R^j] + O(||u||^2) \mbox{ and}  \\
 {\mathcal C}_{p(m^c,u,v)}^r (X_L^i) & = - [ X_L^i, v \cdot X_R] + O(||v||^2),
 \end{align*}
 where the error term $O(||u||^2)$, for example, must be real analytic and vanish quadratically as $u \rightarrow 0$.
 Thus, to first order, the terms whose square sums give ${\mathcal K}^2$ in \eqref{keq} are merely wedge products of $X_L$'s, $X_R$'s, and their first commutators. More precisely, when $i_1,\ldots,i_{\#S_L}$ enumerates $S_L$ and $j_1,\ldots,j_{\#S_R}$ enumerates $S_R$, we have
 \begin{align} 
{\mathbf X}_L \wedge {\mathbf X}_R \wedge  {\mathcal C}_{p(m^c,u,v)}^l {\mathbf X}_R^{S_R} & \wedge {\mathcal C}_{p(m^c,u,v)}^r {\mathbf X}_L^{S_L} = \nonumber \\ 
(-1)^{\#S_L} {\mathbf X}_L \wedge  {\mathbf X}_R \wedge & [u \cdot X_L, X_R^{j_1} ] \wedge \cdots \wedge[u \cdot X_L, X_R^{j_{\#S_R}}] \label{commapp} \\ & \wedge [ X_L^{i_1}, v \cdot X_R] \wedge \cdots \wedge [X_L^{i_{\#S_L}}, v \cdot X_R] \nonumber \\
  & + O ( ||u||^{\#S_L + 1} ||v||^{\#S_R}) + O( ||u||^{\#S_L} ||v||^{\#S_R+1}). \nonumber
 \end{align} 
 Using this calculation and the machinery already produced, to prove Theorem \ref{variable5}, it suffices to establish the following result:
\begin{lemma}
Assume that $\pi_L$ and $\pi_R$ are nondegenerate as in \eqref{leftid} and \eqref{rightid}. Let $A_j := \set{ x \in \R^2}{ 2^{j-1} \leq ||x|| < 2^j}$ as before. then for any compact set $K \subset \onesp$, there is a finite constant $C$ and an integer $j_0$ such that
\[ \left| \set{u \in A_i}{ |\Phi(p(m^c,u,v))| \leq \alpha} \right| \leq C \alpha 2^{i-j} \]
\[ \left| \set{ v \in A_j}{ |\Phi(p(m^c,u,v))| \leq \alpha} \right| \leq C \alpha 2^{-i+j} \]
for all $\alpha \geq 0$ and all $(m^c,u,v) \in K \times A_i \times A_j$ whenever $i,j \leq j_0$. Here $\Phi$ is the function defined by \eqref{phieq}.
\end{lemma} 

\begin{proof}
We must have $\#S_L + \#S_R = 3$, meaning that one set should have cardinality one and the other cardinality two. Both options are completely symmetric, so let us consider the case when $\#S_L = 2$.
For convenience below, given any vector $v := (v_1,v_2)$, we define $v^\perp := (-v_2,v_1)$. For any fixed $v \in \R^2$, let $w$ be any nonzero unit vector such that
\[ {\mathbf X}_L \wedge {\mathbf X}_R \wedge [ X_L^1, v \cdot X_R] \wedge [X_L^2, v \cdot X_R] \wedge [w \cdot X_L, v^{\perp} \cdot X_R] = 0 \]
(notice that $w$ is unique up to sign). It follows that
\begin{align*}
 {\mathbf X}_L & \wedge {\mathbf X}_R \wedge [u \cdot X_L, X_R^1] \wedge [u \cdot X_L, X_R^2] \wedge [w \cdot X_L, v \cdot X_R] \\
  = & ||v||^{-2}  {\mathbf X}_L \wedge {\mathbf X}_R \wedge [u \cdot X_L, v \cdot X_R] \wedge [u \cdot X_L, v^{\perp} \cdot X_R] \wedge [w \cdot X_L, v \cdot X_R] \\
  = & \frac{u \cdot w^\perp}{||v||^{2}}  {\mathbf X}_L \wedge {\mathbf X}_R \wedge [w^\perp \cdot X_L, v \cdot X_R] \wedge [u \cdot X_L, v^{\perp} \cdot X_R] \wedge [w \cdot X_L, v \cdot X_R] \\
  = & -\frac{u \cdot w^\perp}{||v||^{2}}  {\mathbf X}_L \wedge {\mathbf X}_R \wedge [X_L^1, v \cdot X_R] \wedge [u \cdot X_L, v^{\perp} \cdot X_R] \wedge [ X_L^2, v \cdot X_R] \\
  = & \frac{u \cdot w^\perp}{||v||^2} {\mathbf X}_L \wedge {\mathbf X}_R \wedge [ X_L^1, v \cdot X_R] \wedge [X_L^2, v \cdot X_R] \wedge [u \cdot X_L, v^{\perp} \cdot X_R]  \\
  = & \frac{(u \cdot w) (u \cdot w^\perp)}{||v||^2} {\mathbf X}_L \wedge {\mathbf X}_R \wedge [ X_L^1, v \cdot X_R] \wedge [X_L^2, v \cdot X_R] \wedge [w^{\perp} \cdot X_L, v^{\perp} \cdot X_R].
 \end{align*}
By \eqref{leftid}, the wedge product on the final line cannot be zero unless $u = w$ or $u = w^\perp$ (otherwise no such $a_{v,v'}$ could exist since any linear combination of $w$ and $w^\perp$ would return zero when put in the place of $a_{v,v'}$). In particular, the $u$-derivative of this expression is never zero. By compactness of the domain from which $(m^c,u,v)$ is drawn and the continuity of all vector fields, this means
\begin{align*}
& \left| \nabla_u ({\mathbf X}_L \wedge {\mathbf X}_R \wedge {\mathcal C}_{p(m^c,u,v)}^l X_R^1 \wedge {\mathcal C}_{p(m^c,u,v)}^r X_R^2 \wedge {\mathcal C}_{p(m^c,u,v)}^r X_L^1) \right| \\
& + \left| \nabla_u ({\mathbf X}_L \wedge {\mathbf X}_R \wedge {\mathcal C}_{p(m^c,u,v)}^l X_R^1 \wedge {\mathcal C}_{p(m^c,u,v)}^r X_R^2 \wedge {\mathcal C}_{p(m^c,u,v)}^r X_L^2) \right| \geq c ||u|| ||v||
\end{align*}
for some constant $c > 0$, provided $||u||$ and $||v||$ are sufficiently small. By \eqref{commapp}, we also have that there is a finite constant $C$ such that
\begin{align*}
& \left| \nabla_u^2 ({\mathbf X}_L \wedge {\mathbf X}_R \wedge {\mathcal C}_{p(m^c,u,v)}^l X_R^1 \wedge {\mathcal C}_{p(m^c,u,v)}^r X_R^2 \wedge {\mathcal C}_{p(m^c,u,v)}^r X_L^1) \right| \\
& + \left| \nabla_u^2 ({\mathbf X}_L \wedge {\mathbf X}_R \wedge {\mathcal C}_{p(m^c,u,v)}^l X_R^1 \wedge {\mathcal C}_{p(m^c,u,v)}^r X_R^2 \wedge {\mathcal C}_{p(m^c,u,v)}^r X_L^2) \right| \leq C ||v||
\end{align*}
for the same range of $m^c$, $u$, and $v$.  Thus for all sufficiently small annuli, for any fixed values of $m^c$ and $v$, we may cover the annulus $A_i$ by boundedly many balls (independent of the annulus, $m^c$, and $v$), on which there are indices $k_1$ and $k_2$ such that
\[ \left| \frac{\partial}{\partial u_{k_1}}  ({\mathbf X}_L \wedge {\mathbf X}_R \wedge {\mathcal C}_{p(m^c,u,v)}^l X_R^1 \wedge {\mathcal C}_{p(m^c,u,v)}^r X_R^2 \wedge {\mathcal C}_{p(m^c,u,v)}^r X_L^{k_2}) \right| \geq c' 2^{i} ||v||. \]
By the usual Fubini argument, it follows that
\begin{equation} \left| \set{u \in A_i}{ |{\mathcal K}(p(m^c,u,v))| \leq \alpha} \right| \leq C' \alpha 2^{-j} \label{varsub1} \end{equation}
uniformly as desired.

Next, if we let $w'$ be any unit vector such that 
\[ {\mathbf X}_L \wedge {\mathbf X}_R \wedge [u \cdot X_l, X_R^1] \wedge [u \cdot X_L, X_R^{2}] \wedge [u^\perp \cdot X_L, w'\cdot X_R] = 0, \]
we must have that
\begin{align*}
{\mathbf X}_L & \wedge {\mathbf X}_R \wedge [ u \cdot X_L, X_R^1] \wedge [u \cdot X_L, X_R^2] \wedge [ u^\perp X_L , v \cdot X_R] \\
& = (v \cdot (w')^\perp) {\mathbf X}_L \wedge {\mathbf X}_R \wedge [ u \cdot X_L, X_R^1] \wedge [u \cdot X_L, X_R^2] \wedge [ u^\perp X_L , (w')^\perp \cdot X_R ].
\end{align*}
Once again, with the exception of the coefficient $v \cdot (w')^\perp$ vanishing, the wedge product on the right-hand side cannot be zero without contradicting the existence of $b_{u,u^\perp}$ in \eqref{rightid}. Reasoning just as before, we find that
\begin{align*}
& \left| \nabla_v ({\mathbf X}_L \wedge {\mathbf X}_R \wedge {\mathcal C}_{p(m^c,u,v)}^l X_R^1 \wedge {\mathcal C}_{p(m^c,u,v)}^r X_R^2 \wedge {\mathcal C}_{p(m^c,u,v)}^r X_L^1) \right| \\
& + \left| \nabla_v ({\mathbf X}_L \wedge {\mathbf X}_R \wedge {\mathcal C}_{p(m^c,u,v)}^l X_R^1 \wedge {\mathcal C}_{p(m^c,u,v)}^r X_R^2 \wedge {\mathcal C}_{p(m^c,u,v)}^r X_L^2) \right| \geq c ||u||^2
\end{align*}
and
\begin{align*}
& \left| \nabla_v^2 ({\mathbf X}_L \wedge {\mathbf X}_R \wedge {\mathcal C}_{p(m^c,u,v)}^l X_R^1 \wedge {\mathcal C}_{p(m^c,u,v)}^r X_R^2 \wedge {\mathcal C}_{p(m^c,u,v)}^r X_L^1) \right| \\
& + \left| \nabla_v^2 ({\mathbf X}_L \wedge {\mathbf X}_R \wedge {\mathcal C}_{p(m^c,u,v)}^l X_R^1 \wedge {\mathcal C}_{p(m^c,u,v)}^r X_R^2 \wedge {\mathcal C}_{p(m^c,u,v)}^r X_L^2) \right| \leq C \frac{||u||^2}{||v||}
\end{align*}
(where the factor of $||v||^{-1}$ is easily obtained by bounding the second derivatives by $C||u||^2$ and then fixing an upper bound for $||v||$) and consequently that
\begin{equation} \left| \set{v \in A_j}{ |{\mathcal K}(p(m^c,u,v))| \leq \alpha} \right| \leq C' \alpha 2^{-2i-j}. \label{varsub2}
\end{equation}
Now if $i \geq j$, then $\Phi(p(m^c,u,v)) \approx 2^{-i} {\mathcal K}(p(m^c,u,v))$, so \eqref{varsub1} and \eqref{varsub2} imply the lemma by merely replacing $\alpha$ in \eqref{varsub1} and \eqref{varsub2} with $2^i \alpha$. The case $i \leq j$ is obtained in exactly the same manner by fixing $\#S_L = 1$ and $\#S_R = 2$ (and effectively interchanging the roles of $u$ and $v$).
\end{proof}

\subsection{The maximal complex quadratic submanifold}

In the case of the bilinear functional \eqref{b2d} corresponding to the operator \eqref{modelmaxml}, if we complexify the manifolds $\onesp, \Lsp, $ and $\Rsp$, we are naturally led to the following functional representing an integral over a quadratic four-dimensional submanifold of $\R^{10}$, which corresponds to the bilinear functional for \eqref{modelmaxmlc} written in real coordinates:
\begin{align}
 B (f,g) := \int_\Omega f \left( \vphantom{\frac{1}{2}} \right. \! \! x_1 + t_1, \ldots, x_4 + t_4, & x_5 + \frac{1}{2} (t_1^2 - t_2^2), x_6 + t_1 t_2, \nonumber  \\
   & x_7 + t_1 t_3 - t_2 t_4, x_8 + t_1 t_4 + t_2 t_3, \label{bcomplex} \\ 
   & x_9 + \frac{1}{2} (t_3^2 - t_4^2), x_{10} + t_3 t_4 \left. \vphantom{\frac{1}{2}} \! \! \right) g(x) dt dx. \nonumber
 \end{align}
 We take the following definitions of $X_L^i$ and $X_R^j$:
 \begin{align*}
 X_L^1 := & \frac{\partial}{\partial t_1} - \frac{\partial}{\partial x_1} - t_1 \frac{\partial}{\partial x_5} - t_2 \frac{\partial}{\partial x_6} - t_3 \frac{\partial}{\partial x_7} - t_4 \frac{\partial}{\partial x_8}, \\
X_L^2 := & \frac{\partial}{\partial t_2} - \frac{\partial}{\partial x_2} + t_2 \frac{\partial}{\partial x_5} - t_1 \frac{\partial}{\partial x_6} + t_4 \frac{\partial}{\partial x_7} - t_3 \frac{\partial}{\partial x_8}, \\
X_L^3 := & \frac{\partial}{\partial t_3} - \frac{\partial}{\partial x_3} - t_1 \frac{\partial}{\partial x_7} - t_2 \frac{\partial}{\partial x_8} - t_3 \frac{\partial}{\partial x_9} - t_4 \frac{\partial}{\partial x_{10}}, \\
X_L^4 := &  \frac{\partial}{\partial t_4} - \frac{\partial}{\partial x_4} + t_2 \frac{\partial}{\partial x_7} - t_1 \frac{\partial}{\partial x_8} + t_4 \frac{\partial}{\partial x_9} - t_3 \frac{\partial}{\partial x_{10}}, \\
X_R^j := & \frac{\partial}{\partial t_j}, \ j=1,\ldots,4.
\end{align*}
In this case, we need not even sum over all $S_L$ and $S_R$ in \eqref{keq} to estimate $\mathcal K$ from below well enough to establish boundedness of the sublevel set operators \eqref{sublevels}. In particular, we need only sum over those $S_L$ and $S_R$ in which the indices $1$ and $2$ are either both omitted or occur simultaneously as a pair, and likewise with the indices $3$ and $4$. This gives all the determinants a block complex structure which is easy to evaluate: if $z_{jk} = a_{jk} + i b_{jk}$, where $i$ is the imaginary unit, then 
\[ \det \left[ \begin{array}{rrcrr} a_{11} & -b_{11} & \cdots & a_{1n} & -b_{1n} \\
b_{11} & a_{11} & \cdots & b_{1n} & a_{1n} \\
\vdots & \vdots & \ddots & \vdots & \vdots \\
a_{n1} & -b_{n1} & \cdots & a_{nn} & -b_{nn} \\
b_{n1} & a_{n1} & \cdots & b_{nn} & a_{nn}
\end{array} \right] = \left| \det \left[ \begin{array}{ccc} z_{11} & \cdots & z_{1n} \\ \vdots & \ddots & \vdots \\ z_{n1} & \cdots & z_{nn} \end{array} \right] \right|^2. \]
In particular, this reduction brings us back to the same calculations encountered in \eqref{keq2d}; the only difference is that the entries of the various determinants are allowed to be complex. The end result is that the sublevel set operators $W_{\alpha,m^c}(g,f)$ are dominated (modulo the multiplication of $\alpha$ by a constant factor) by the complexified  sublevel set operator
\[ \int_{{\C}^2 \times {\C}^2} \chi_{ | z_1 w_2 - z_2 w_1|^2 \leq \alpha} f(z_1,z_2) g(w_1,w_2) \left| d z_1 \wedge d \overline{z_1} \wedge \cdots \wedge d w_2 \wedge d \overline{w_2} \right|, \]
where $z_j$ and $w_j$ are now, of course, complex. If $A_j \subset \C^2$ is the complex annulus $\set{(z_1,z_2) \in \C^2}{ 2^{j-1} \leq \sqrt{ |z_1|^2 + |z_2|^2} < 2^{j}}$, then just as in \eqref{msublev1} and \eqref{msublev2}, 
\begin{align*} \int \chi_{ | z_1 w_2 - z_2 w_1|^2 \leq \alpha} & \chi_{E^r \cap A_k} (z_1,z_2) \chi_{E^l \cap A_j} (w_1,w_2) \left| d z_1 \wedge d \overline{z_1} \wedge \cdots \wedge d w_2 \wedge d \overline{w_2} \right| \\
& \leq C \min \{ 2^{2j-2k} \alpha |E^r \cap A_k|, 2^{-2j + 2k} \alpha |E^l \cap A_j |, \}
\end{align*}
which gives just as in the real case that the bilinear functional \eqref{bcomplex} satisfies \[ |B(f,g)| \leq C ||f||_{q_L} ||g||_{q_R} \]
whenever $(\frac{1}{q_L},\frac{1}{q_R})$ belongs to the interior of the triangle with vertices $(1,0),$ $(0,1)$, and $(\frac{5}{8}, \frac{5}{8})$. Just as in the previous section, the constraint that $\Omega$ in \eqref{bcomplex} be compact can be relaxed by H\"{o}lder to establish boundedness of \eqref{modelmaxmlc}.

\subsection{The 3D harmonic quadratic surface in $\R^8$ and generalizations}

In the case of the bilinear functional corresponding to \eqref{modelharmonic}, we have
\begin{align*}  B(f,g) := \int_{\Omega} f & \left( x_1 + t_1,x_2 + t_2, x_3 + t_3,  x_4 + \frac{t_1^2}{2} - \frac{t_2^2}{2}, \right. \\ 
& \left. \ \ x_5 + \frac{t_2^2}{2} - \frac{t_3^2}{2}, x_6 +  t_1 t_2, x_7 +  t_2 t_3, x_8 +  t_1 t_3 \right) g(x) dx dt.
\end{align*}
Calculating just as before using the vector fields
\begin{align*}
X_L^1 & := \frac{\partial}{\partial t_1} - \frac{\partial}{\partial x_1} + t_1 \frac{\partial}{\partial x_4} + t_2 \frac{\partial}{\partial x_6} + t_3 \frac{\partial}{\partial}{x_8}, \\
X_L^2 &:= \frac{\partial}{\partial t_2} - \frac{\partial}{\partial x_2} - t_2 \frac{\partial}{\partial x_4} + t_2 \frac{\partial}{\partial x_5} + t_1 \frac{\partial}{\partial x_6} + t_3 \frac{\partial}{\partial}{x_7}, \\
X_L^3 &:= \frac{\partial}{\partial t_3} - \frac{\partial}{\partial x_3} - t_3 \frac{\partial}{\partial x_5} + t_2 \frac{\partial}{\partial x_7} + t_1 \frac{\partial}{\partial}{x_8}, \\
X^j_R & := \frac{\partial}{\partial t_j}, \ j = 1,2,3,
\end{align*}
we come to the conclusion that the function $\Phi$ given by \eqref{phieq} governing the relevant sublevel set functional is given by the formula
\[\Phi (p) =  ||t^l - t^c||^2 ||t^r - t^c||^2 - ((t^l - t^c) \cdot (t^r - t^c))^2. \]
Also in agreement with previous cases, we may assume without loss of generality that $t^c = 0$ and we let $A_i$ be the annulus $\set{ t \in \R^3}{ 2^{i-1} \leq ||t|| < 2^i}$. For fixed, nonzero $y \in \R^3$, we observe that the set
\[ \set{x \in\R^3}{ ||x||^2 ||y||^2 - (x \cdot y)^2 \leq \alpha} \]
consists of exactly those points at a distance $\alpha^{1/2} ||y||^{-1}$ to the line through the origin with direction $y$. Therefore using Fubini's theorem exactly as in the earlier case of the 2-dimensional surface in $\R^5$, we conclude that
\begin{align*}
\int_{ ||t^l||^2 ||t^r||^2 - (t^l \cdot t^r)^2 \leq \alpha} &  \chi_{E^l \cap A_i}(t^l) \chi_{E^r \cap A_j}(t^r) dt^l dt^r \\
& \leq C \alpha \min\{ 2^{i - 2j} |E^r \cap A_j|, 2^{-2i + j} |E^l \cap A_i| \},
\end{align*}
which implies by \eqref{marc} with $\theta = \frac{1}{2}$ that
\begin{align*} \int_{ ||t^l||^2 ||t^r||^2 - (t^l \cdot t^r)^2 \leq \alpha} &  \chi_{E^l}(t^l) \chi_{E^r} (t^r) dt^l dt^r \\ & \leq C' \alpha \left( \sum_{i \in \Z} 2^{-i} |E^l \cap A_i| \right)^{\frac{1}{2}} \left( \sum_{j \in \Z} 2^{-j} |E^r \cap A_j| \right)^{\frac{1}{2}}. 
\end{align*}
This inequality, by itself, does not immediately imply boundedness of \eqref{sublevels}, but the desired inequality follows from the observation that
\begin{equation} \left( \sum_{i \in \Z} 2^{-i} |E^l \cap A_i| \right)^{\frac{1}{2}} \leq C'' |E^l|^{\frac{1}{3}} \label{wrtwum} \end{equation}
and likewise for the sum over $j$. The observation itself is a consequence of the fact that $|E^l \cap A_i| \lesssim 2^{3i}$; breaking the sum into parts $i \leq i_0$ and $i > i_0$ and estimating separately gives
\[ \sum_{i \leq i_0} 2^{-i} |E^l \cap A_i| \lesssim \sum_{i \leq i_0} 2^{-i} 2^{3i} \lesssim 2^{2 i_0}, \]
\[ \sum_{i > i_0} 2^{-i} |E^l \cap A_i| \leq 2^{-i_0} \sum_{i \in \Z} |E^l \cap A_i| \leq 2^{-i_0} |E^l|. \]
Therefore
\[\left( \sum_{i \in \Z} 2^{-i} |E^l \cap A_i| \right)^{\frac{1}{2}} \lesssim \left( 2^{2i_0} + 2^{-i_0} |E^l| \right)^{\frac{1}{2}}. \]
Optimizing the choice of $i_0$ gives \eqref{wrtwum}.
Consequently
\begin{align*}
\int_{ ||t^l||^2 ||t^r||^2 - (t^l \cdot t^r)^2 \leq \alpha} &  \chi_{E^l}(t^l) \chi_{E^r}(t^r) dt^l dt^r \leq C \alpha |E^l|^{\frac{1}{3}} |E^r|^{\frac{1}{3}}
\end{align*}
and
\[ | B(f,g)| \lesssim ||f||_{q_L} ||g||_{q_R} \]
for $(\frac{1}{q_L}, \frac{1}{q_R})$ in the interior of the triangle with vertices $(1,0), (0,1),$ and $(\frac{8}{13}, \frac{8}{13})$. Using H\"{o}lder's inequality as in previous cases to allow $\Omega := \R^{8} \times [-1,1]^3$ and then applying Marcinkiewicz interpolation gives the boundedness of \eqref{modelharmonic} as asserted by Theorem \ref{modelthm}.

%%%\[ Q (x,y) := \left[ \begin{array}{c} x_1 y_1 - x_2 y_2 \\ x_2 y_2 - x_3 y_3 \\ x_1 y_2 + x_2 y_1 \\ x_2 y_3 + x_3 y_2 \\ x_1 y_3 + x_3 y_1 \end{array} \right] \qquad {\mathcal K} = (|| x||^2 + ||y||^2)^{\frac{1}{2}}  \left( ||x||^2 ||y||^2 - (x \cdot y)^2 \right) \]
%%%One-dimensional fibers (since dimension of $\threesp$ is $17$ and $\Pi$ maps into space of dimension $16$).

\subsection{Uneven half-dimensional averages and generalizations}
The final calculations deal with the operator \eqref{modelasym} and its geometric variants described by Theorem \ref{variableasym}.
In the case of \eqref{modelasym}, the relevant bilinear functional is given by
\[ B(f,g) := \int_{\Omega} f(x, y + t x) g( t, y) dx dt dy. \]
Here $\ell = d_R$ and $d_L = 1$. Using Theorem \ref{bilinearthm}, we have that
\[ \Phi_Q(t,x) := \frac{ \mathrm{Vol}( \{ t e_i \}_{i=1}^{d_L}, \{x\})}{ (|x|^2 + |t|^2)^{\frac{1}{2}}} \approx \frac{\max \{ |t|^{d_R}, |x| |t|^{d_R -1} \}}{(|x|^2 + |t|^2)^{\frac{1}{2}}} \approx |t|^{d_R-1}. \]
The corresponding sublevel set inequality we come to is simply that
\[ \int_{|t|^{d_R-1} \leq \alpha} \chi_{E^l}(t) \chi_{E^r}(x) dx dt \lesssim \alpha^{\frac{1}{d_R - 1}} |E^r|. \]
This estimate is trivial to obtain by integrating over $t$ first. Since we also have $|B(f,g)| \leq C ||f||_1 ||g||_\infty$ and $|B(f,g)| \leq C ||f||_\infty ||g||_1$, it follows that
\[ |B(f,g)| \leq C ||f||_{q_L} ||g||_{q_R} \]
provided that $(\frac{1}{q_L}, \frac{1}{q_R})$ belongs to the interior of the triangle with vertices $(1,0)$, $(0,1)$, and
$( \frac{d_R +1}{d_R+2}, \frac{2}{d_R+2})$.

Regarding Theorem \ref{variableasym}, 
the calculation proceeds in much the same fashion as the proof of Theorem \ref{variable5}. In particular, we continue to use the coordinate system
\[ p(m^c,u,v) := (\exp(u \cdot X_L) m^c, m^c, \exp(v \cdot X_R) m^c). \]
In the definition \eqref{keq} of $\mathcal K$, we must have either $\#S_L = 0, \#S_R = d_R$ or $\#S_L = 1, \#S_R = d_R - 1$. In the former case, using \eqref{commapp}, we have
\begin{align}
X_L & \wedge {\mathbf X}_R \wedge {\mathcal C}^l_{p(m^c,u,v)} {\mathbf X}_R  \nonumber \\
 & = u^{d_R} X_L \wedge {\mathbf X}_R \wedge [X_L, X_R^1] \wedge \cdots \wedge [X_L,X_R^{d_R}] + O(|u|^{d_R+1}), \label{asym1}
\end{align}
and in the latter case we must have that
\begin{align}
X_L  \wedge {\mathbf X}_R \wedge   {\mathcal C}^l_{p(m^c,u,v)} & {\mathbf X}_R^{\{1,\ldots,\widehat{j},\ldots,d_R\}} \wedge {\mathcal C}_{p(m^c,u,v)}^r (X_L)  \nonumber \\
  = -u^{d_R-1} & X_L \wedge {\mathbf X}_R \wedge [X_L, X_R^1] \wedge \nonumber \\  \cdots & \wedge \widehat{[X_L,X_R^j]} \wedge \cdots \wedge [X_L,X_R^{d_R}] \wedge [ X_L, v \cdot X_R] \nonumber \\
  & + O(|u|^{d_R}) + O(|u|^{d_R-1} ||v||) \nonumber \\
  =  (-1)^{d_R-j+1} &  v_j u^{d_R-1}  X_L \wedge {\mathbf X}_R\wedge [X_L, X_R^1] \wedge  \cdots  \wedge [X_L,X_R^{d_R}] \label{asym2} \\
  & + O(|u|^{d_R}) + O(|u|^{d_R-1} ||v||). \nonumber
\end{align}
Assuming the nondegeneracy condition
\[ X_L \wedge {\mathbf X}_R \wedge [X_L, X_R^1] \wedge \cdots \wedge [X_L,X_R^{d_R}] \neq 0, \]
 when we sum the squares of \eqref{asym1} and \eqref{asym2} (in the latter case, summing over $j$), we must have that
\[ \Phi(p(m^c,u,v)) \geq C |u|^{d_R-1} \]
provided that $|u|$ and $||v||$ are sufficiently small. Thus the sublevel set estimate follows exactly as in the model case.

\section{Appendix: Regularity of Fiber Measures}
\label{solsec}

%\subsection{Fiber regularity as a generalized counting problem}
%{\bf CHANGE THE M HERE BECAUSE IT'S REALLY $\threesp$!!!!!!!!!!!!!}

In this section, we take up the last remaining technical issue, namely that the measure $\mu_0$, defined by \eqref{densityeq}, on the nondegenerate part of the fibers of $\Pi$, satisfy the regularity condition
\begin{equation} \mu_0(B_r(x)) \lesssim r^{\kappa} \label{regasmp} \end{equation}
for all $r$ sufficiently small and all $x$ in some compact subset of $\Lsp \times \Rsp$. This is also the only place in this paper where the real analyticity assumption becomes important; in particular, if one could establish \eqref{regasmp} directly by other means (which seems likely to be possible in the cases of Theorems \ref{variable5} and \ref{variableasym}), then real analyticity would no longer be necessary.

The interesting feature of \eqref{regasmp} is that, under the assumption of real analyticity, the estimate is closer to a counting statement than it is to an integral estimate. Thus, in switching from the Inverse Function Theorem to the Implicit Function Theorem in the method of refinements, while B\'{e}zout's theorem is no longer directly applicable, the results of this section which take its place are not so different in spirit. The proof of \eqref{regasmp} begins with a general lemma which highlights the fundamentally discrete nature of integrals over fibers.
\begin{lemma}
Let ${\mathcal M}$, ${\mathcal X}$, and ${\mathcal Z}$ be smooth manifolds of dimension $n$, $k$, $n-k$, respectively. Suppose that ${\mathcal Z}$ is equipped with a measure of smooth density $\mu_{\mathcal Z}$. If $\Pi : {\mathcal M} \rightarrow {\mathcal X}$ and  $\rho : {\mathcal M} \rightarrow {\mathcal Z}$ are smooth maps, then for any $x \in {\mathcal X}$ and any Borel measurable set $E \subset {\mathcal M}$ such that $d \Pi$ is surjective at all points $m \in E$,
\[ \int_{\Pi^{-1} (x)} \chi_E ~ \rho^*( d \mu_{\mathcal Z}) = \int_{\rho(E)} N_E(x,z) d \mu_{\mathcal Z} (z) \]
where $N_E(x,z)$ is the number of solutions $m \in E$ of the system $\Pi(m) = x$, $\rho(m) = z$ at which $\ker d \Pi|_m \cap \ker d \rho|_m$ is trivial. \label{countlem}
\end{lemma}
\begin{proof}
Without loss of generality, we may assume that $d \Pi$ is surjective at every point of ${\mathcal M}$. The Implicit Function Theorem guarantees that $\Pi^{-1}(x)$ is a smooth $(n-k)$-dimensional submanifold (when nonempty). Since $E$ is Borel, the restriction of $E$ to $\Pi^{-1}(x)$ will also be Borel. The mapping $\rho$, when restricted to the submanifold $\Pi^{-1}(x)$ will have surjective differential at exactly those points $m$ at which $\ker d \Pi|_m \cap \ker d \rho|_m$ is trivial. By Sard's Lemma, the integral over the image under $\rho$ of the complement of this set (namely, the set where the intersection of kernels is nontrivial) will have $\mu_{\mathcal Z}$-measure zero, so we may assume without loss of generality that $E$ also does not contain any such points. When the differential $d \rho$ is surjective, $\rho$ is locally bijective and the Inverse Function Theorem may be applied. In particular, for any point $m \in E$ and any sufficiently small open set $U$ containing $m$, we must have that
\[ \int_{\Pi^{-1}(x)} \varphi_U \chi_{E \cap U} \rho^*( d \mu_{\mathcal Z}) = \int_{\rho(U)}  \varphi_{U}(\rho^{-1}(z)) \chi_{E \cap U} (\rho^{-1}(z)) d \mu_{\mathcal Z}(z) \]
for any smooth function $\varphi_U$ supported on $U$. Taking $\varphi_U$ to be elements of a partition of unity on ${\mathcal M}$ subordinate to the neighborhoods $U$ and summing over $U$ gives the conclusion of the lemma.
\end{proof}
In light of Lemma \ref{countlem}, the usefulness of real analyticity comes in to sharp focus: fundamental work of Gabrielov \cite{gabrielov1968} establishes that, when that the manifolds ${\mathcal M}, {\mathcal Z}$, and ${\mathcal X}$ and mappings $\Pi$ and $\rho$ are all real analytic and $E$ is contained in a fixed compact set $K$, $N_{E}(x,z)$ is uniformly bounded by some constant $N$ independent of $E$, $x$ and $z$. Thus the integral of $\rho^*(d \mu_{\mathcal Z})$ over partial fibers $\Pi^{-1}(x) \cap E$ is controlled by a bounded constant times the $\mu_{\mathcal Z}$-measure of the projection $\rho(E)$.

To apply this insight to the case of $\mu_0$, let $U_\Lsp \subset \Lsp$ and $U_\Rsp \subset \Rsp$ be open sets and let $\rho_\Lsp : U_{\Lsp} \rightarrow \R^{n_L}$ and $\rho_\Rsp : U_{\Rsp} \rightarrow \R^{n_R}$ be coordinate charts satisfying 
\[ \dist(x_L,x_L') \approx || \rho_\Lsp(x_L) - \rho_\Lsp(x_L')|| \mbox{ and } \dist(x_R,x_R') \approx || \rho_\Rsp(x_R) - \rho_\Rsp(x_R')|| \]
for some finite implicit constants and every $x_L, x_L' \in U_\Lsp$ and $x_R,x_R' \in U_\Rsp$. For convenience, let $U_\onesp := \pi_L^{-1} U_L \cap \pi_R^{-1} \cap U_R$.  If we also define
\[ \rho_0(m) :=  (\rho_\Lsp \circ \pi_L(m), \rho_\Rsp \circ \pi_R(m)), \]
then $\mu_\onesp$ being subordinate to the Riemannian measure on $\mathcal M$ means in particular that there is a uniform constant $C$ such that
\begin{equation}
\begin{split}
 \mu_{\onesp} &  ( V_1,\ldots,V_{n_L + n_R - \ell}) \\
 & \leq C \left[ \sum_{\#S = \ell} \left| \det \left[ e_{S_1} , \ldots, e_{S_\ell}, d \rho_0 ( V_1),\ldots,d \rho_0 (V_{n_L + n_R - \ell}) \right] \right|^2 \right]^\frac{1}{2}
 \end{split} \label{mdens}
 \end{equation}
 for any vectors $V_1,\ldots,V_{n_L + n_R - \ell}$ at any point in $U_\onesp$, 
where $S$ ranges over all cardinality $\ell$ subsets of any orthonormal basis of $\R^{n_L + n_R}$.

Now for any choice $\sigma$ of a $\kappa$-dimensional subset of $\R^{n_L + n_R}$ spanned by a subset of the basis vectors $e_k$, let
 $\rho_\sigma : (\threesp \cap U_\onesp^3) \rightarrow \R^\kappa$ be defined by the map
\[ \rho_\sigma(m^l,m^c,m^r) := P_\sigma \rho_0 ( \pi^c (m^l,m^c,m^r)), \] 
 where $P_\sigma$ is projection onto the span of the unit vectors determined by $\sigma$. If $d {\mathcal Z}$ is the Lebesgue measure on $\R^\kappa$, we must have the following identity for the pullback measure $\rho_\sigma^*(d {\mathcal Z})$ on the incidence manifold $\threesp \cap U_\onesp^3$:
\begin{equation} 
\begin{split}
\rho_\sigma^* & (d {\mathcal Z}) ( V_1,\ldots,V_\kappa) \\
 & = \left| \det \left[ e_{i_1} \ldots, e_{i_{n_L + n_R - \kappa}}, d \rho_0 d \pi^c(V_1),\ldots, d \rho_0 d \pi^c (V_\kappa) \right] \right|,
 \end{split} \label{pullb}
\end{equation}
where $e_{i_1},\ldots,e_{i_{n_L+n_R-\kappa}}$ are precisely those basis vectors not belonging to $\sigma$.
In comparing \eqref{pullb} to the right-hand side of \eqref{mdens}, note that $\ell \leq n_L + n_R - \kappa$, so that there are more elements of the orthonormal basis appearing on the right-hand side of \eqref{pullb}.

Recall the definition \eqref{densityeq} of the density $\mu_0$. We may assume that the constant $C_\inc$ is smooth and nonvanishing. Likewise we may assume that $\inc_p^l X^j_R$, $\inc_p^r X_L^i$, and $\inc_p v$ are all smooth and have bounded norms. At any particular point, we may also assume that the span of these vectors is equal to the span of some collection of orthonormal basis vectors of the same cardinality (since if this were not the case the density $\mu_0$ would trivially vanish). Therefore we have the inequality
\[ \mu_0({\mathbf P}) \leq C' \sum_{\sigma} \rho_\sigma^*(d {\mathcal Z}) ({\mathbf P}) \]
when ${\mathbf P}$ is any decomposable $\kappa$-multivector field generated by fibers of $\Pi$.
In particular, we must have the integral inequality
\begin{align*}
\mu_0(B_r& (x_L,x_R)) = \\
 & \int_{\Pi^{-1}(x_L,x_R) \cap (U_\onesp)^3}  \chi_{(\dist(\pi^c_L(m),x_L))^2 + (\dist(\pi^c_R(m),x_R))^2 \leq r^2} d \mu_0(m) \\
 & \leq C \sum_{\sigma} \int_{\Pi^{-1}(x_L,x_R) \cap (U_\onesp)^3} \chi_{\dist(\rho_\sigma^*(m), P_\sigma(x_L,x_R)) \leq r} \rho_\sigma^*(d {\mathcal Z}).
 \end{align*}
The right-hand side can now be estimated by Lemma \ref{countlem}.  If we assume that the number of nondegenerate solutions in $\threesp \cap U_{\mathcal M}^3$ of the system
\[ \Pi(m) = x, \rho_\sigma (m) = z \]
is bounded for all $m$, which by a result of Gabrielov \cite{gabrielov1968} will be the case when the closure of $U_{\mathcal M}$ is compact and the manifolds and mappings are all real analytic, then it must follow that
\[  \int_{\Pi^{-1}(x_L,x_R) \cap (U_\onesp)^3}  \chi_{(\dist(\pi^c_L(m),x_L))^2 + (\dist(\pi^c_R(m),x_R))^2 \leq r^2} d \mu_0(m) \lesssim r^\kappa \]
since the projection of the set $(\dist(\pi^c_L(m),x_L))^2 + (\dist(\pi^c_R(m),x_R))^2 \leq r^2$ via $\rho_\sigma$ is contained in a Euclidean ball of radius comparable to $r$.

%%%\subsection{Regularity of $\mu_\onesp$}
%%%
%%%Regularity of $\mu_\onesp$ is much easier to establish than the regularity of $\mu_0$. In particular, real analyticity is not needed. The reason for this is that we can explicitly assume that $\onesp$ is an embedded submanifold of $\Lsp \times \Rsp$ and that its measure is dominated by the Hausdorff measure (induced by the Riemannian metrics on $\Lsp$ and $\Rsp$). In particular, near any point $m \in {\mathcal M}$, we may choose coordinates on $\Lsp$ and $\Rsp$ near $\pi_L(m)$ and $\pi_R(m)$, respectively, and find a (possibly smaller) neighborhood $U_{\mathcal M}$ of $m$ in which ${\mathcal M}$ is a graph over some domain with compact closure. As a graph, the distance between two points in $U_{\mathcal M}$ must be comparable to the distance of their independent variables. Consequently, if we look at the intersection of $U_{\mathcal M}$ with any ball in $\Lsp \times \Rsp$ of radius $r$, the corresponding set of independent variables must have diameter comparable to $r$ in the domain of the graph, so the measure of the set in the domain is comparable to $r^{n_L + n_R - \ell}$. Since $\mu_\onesp$ is dominated by the Hausdorff measure, if $B_r$ is any ball of radius $r$ in $\Lsp \times \Rsp$, it must be the case that 
%%%\[ \mu_\onesp(B_r \cap U_{\mathcal M}) \lesssim r^{n_L + n_r - \ell}. \]

\bibliography{mybib}

\end{document}